\documentclass{article}%
\usepackage{amssymb}
\usepackage{makeidx}
\usepackage{amsfonts}
\usepackage{amsmath}
\usepackage{graphicx}
\usepackage{theorem}
\usepackage{color}
\usepackage{color}%
\providecommand{\U}[1]{\protect\rule{.1in}{.1in}}
\providecommand{\U}[1]{\protect\rule{.1in}{.1in}}
\newtheorem{theorem}{Theorem}
\newtheorem{acknowledgement}[theorem]{Acknowledgement}

\newtheorem{corollary}[theorem]{Corollary}

\newtheorem{definition}[theorem]{Definition}
{\theorembodyfont{\rmfamily}

}

\newtheorem{lemma}[theorem]{Lemma}
\newtheorem{notation}[theorem]{Notation}

\newtheorem{proposition}[theorem]{Proposition}
{\theorembodyfont{\rmfamily}
\newtheorem{remark}[theorem]{Remark}
}

\newenvironment{proof}[1][Proof]{\noindent\textbf{#1} }{\ \rule{0.5em}{0.5em}}
\begin{document}

\title{On the best constant in Gaffney inequality}
\date{}
\author{G. CSATO,\quad B. DACOROGNA\quad and\quad S. SIL\smallskip\\G.C.: Facultad de Ciencias Fisicas y Matematicas\\Universidad de Concepcion, Concepcion, Chile\\gy.csato.ch@gmail.com\smallskip\\B.D.: Section de Math\'{e}matiques\\EPFL, 1015 Lausanne, Switzerland\\bernard.dacorogna@epfl.ch\smallskip\\S.S.: Section de Math\'{e}matiques\\EPFL, 1015 Lausanne, Switzerland\\swarnendu.sil@epfl.ch}
\maketitle

\begin{abstract}
We discuss the value of the best constant in Gaffney inequality namely%
\[
\left\Vert \nabla\omega\right\Vert _{L^{2}}^{2}\leq C\left(  \left\Vert
d\omega\right\Vert _{L^{2}}^{2}+\left\Vert \delta\omega\right\Vert _{L^{2}%
}^{2}+\left\Vert \omega\right\Vert _{L^{2}}^{2}\right)
\]
when either $\nu\wedge\omega=0$ or $\nu\,\lrcorner\,\omega=0$ on
$\partial\Omega.$

\end{abstract}

\section{Introduction}

We start by recalling Gaffney inequality for vector fields. Let $\Omega
\subset\mathbb{R}^{n}$ be a bounded open smooth set and $\nu$ be the outward
unit normal to $\partial\Omega.$ Then there exists a constant $C=C\left(
\Omega\right)  >0$ such that for every vector field $u\in W^{1,2}\left(
\Omega;\mathbb{R}^{n}\right)  $%
\[
\left\Vert \nabla u\right\Vert _{L^{2}}^{2}\leq C\left(  \left\Vert
\operatorname{curl}u\right\Vert _{L^{2}}^{2}+\left\Vert \operatorname{div}%
u\right\Vert _{L^{2}}^{2}+\left\Vert u\right\Vert _{L^{2}}^{2}\right)
\]
where, on $\partial\Omega,$ either $\nu\wedge u=0$ (i.e. $u$ is parallel to
$\nu$ and we write then $u\in W_{T}^{1,2}\left(  \Omega;\mathbb{R}^{n}\right)
$) or $\nu\,\lrcorner\,u=0$ (i.e. $u$ is orthogonal to $\nu$ and we write then
$u\in W_{N}^{1,2}\left(  \Omega;\mathbb{R}^{n}\right)  $). In the context of
differential forms (identifying $1-$forms with vector fields) this generalizes
to (using the notations of \cite{Csato-Dac-Kneuss(livre)} which are summarized
in the next section) the following theorem (for references see below).

\begin{theorem}
[Gaffney inequality]Let $0\leq k\leq n$ and $\Omega\subset\mathbb{R}^{n}$ be a
bounded open $C^{2}$ set. Then there exists a constant $C=C\left(
\Omega,k\right)  >0$ such that%
\[
\left\Vert \nabla\omega\right\Vert _{L^{2}}^{2}\leq C\left(  \left\Vert
d\omega\right\Vert _{L^{2}}^{2}+\left\Vert \delta\omega\right\Vert _{L^{2}%
}^{2}+\left\Vert \omega\right\Vert _{L^{2}}^{2}\right)
\]
for every $\omega\in W_{T}^{1,2}\left(  \Omega;\Lambda^{k}\right)  \cup
W_{N}^{1,2}\left(  \Omega;\Lambda^{k}\right)  .$
\end{theorem}

The aim of this article is to study the best constant in such inequality. We
therefore define%
\begin{equation}
C_{T}\left(  \Omega,k\right)  =\sup_{\omega\in W_{T}^{1,2}\setminus\left\{
0\right\}  }\left\{  \frac{\left\Vert \nabla\omega\right\Vert ^{2}}{\left\Vert
d\omega\right\Vert ^{2}+\left\Vert \delta\omega\right\Vert ^{2}+\left\Vert
\omega\right\Vert ^{2}}\right\}  \label{definition de C T}%
\end{equation}%
\begin{equation}
C_{N}\left(  \Omega,k\right)  =\sup_{\omega\in W_{N}^{1,2}\setminus\left\{
0\right\}  }\left\{  \frac{\left\Vert \nabla\omega\right\Vert ^{2}}{\left\Vert
d\omega\right\Vert ^{2}+\left\Vert \delta\omega\right\Vert ^{2}+\left\Vert
\omega\right\Vert ^{2}}\right\}  \label{definition de C N}%
\end{equation}
where $\left\Vert \cdot\right\Vert $ stands for the $L^{2}-$norm. It is easy
to see (cf.\thinspace Proposition \ref{Proposition CT et CN plus grand que 1})
that $C_{T}\left(  \Omega,k\right)  ,C_{N}\left(  \Omega,k\right)  \geq1.$ The
cases $k=0$ and $k=n$ are trivial and we always have then that $C_{T}%
=C_{N}=1;$ so the discussion deals with the case $1\leq k\leq n-1.$ In our
main result, we need the concept of $k-$convexity (cf.\quad Definition
\ref{Definition k-convexite}), which generalizes the usual notion of
convexity, which corresponds to the case of $1-$convexity, while $\left(
n-1\right)  -$convexity means that the mean curvature of $\partial\Omega$ is
non-negative. Our main result (cf.\thinspace Theorem
\ref{Theoreme equiv pour CT}) shows that the following conditions are
equivalent.\smallskip

(i) $C_{T}\left(  \Omega,k\right)  =1$ (respectively $C_{N}\left(
\Omega,k\right)  =1$).\smallskip

(ii) $\Omega$ is $\left(  n-k\right)  -$convex (respectively $\Omega$ is
$k-$convex).\smallskip

(iii) The sharper version of Gaffney inequality holds namely $\forall
\,\omega\in W_{T}^{1,2}\left(  \Omega;\Lambda^{k}\right)  $ (respectively
$\forall\,\omega\in W_{N}^{1,2}\left(  \Omega;\Lambda^{k}\right)  $)%
\[
\left\Vert \nabla\omega\right\Vert ^{2}\leq\left\Vert d\omega\right\Vert
^{2}+\left\Vert \delta\omega\right\Vert ^{2}.
\]

(iv) The respective supremum is not attained.\smallskip

(v) $C_{T}$ (respectively $C_{N}$) is scale invariant, namely, for every $t>0$%
\[
C_{T}\left(  t\,\Omega,k\right)  =C_{T}\left(  \Omega,k\right)  .
\]

The result that $\Omega$ is $\left(  n-k\right)  -$convex implies that
$C_{T}\left(  \Omega,k\right)  =1$ was already observed by Mitrea \cite{Mitrea
2001}.\smallskip

We also show that for general non-convex domains the constants $C_{T}\left(
\Omega,k\right)  $ and $C_{N}\left(  \Omega,k\right)  $ can be arbitrarily
large (see Proposition \ref{Proposition C arbitr grand}).\smallskip

The smoothness of the domain is essential in the previous discussion. We
indeed prove (see Theorem \ref{Theoreme polytopes}) that if $\Omega$ is a
polytope (convex or not), or even more generally a set whose boundary is
composed only of hyperplanes, then%
\[
\left\Vert \nabla\omega\right\Vert ^{2}=\left\Vert d\omega\right\Vert
^{2}+\left\Vert \delta\omega\right\Vert ^{2},\quad\forall\,\omega\in C_{T}%
^{1}\left(  \overline{\Omega};\Lambda^{k}\right)  .
\]
i.e. Gaffney inequality is, in fact, an equality with constant $C_{T}\left(
\Omega,k\right)  =1$ (the result is also valid with $T$ replaced by $N$). This
fact was already observed in \cite{Costabel Dauge 1999} when $n=3$ and $k=1$
for polyhedra.\smallskip

We now comment briefly on the history of Gaffney inequality. The proof goes
back to Gaffney \cite{Gaffney1}, \cite{Gaffney2} for manifolds without
boundary and for manifolds with boundary to Friedrichs \cite{Friedrichs},
Morrey \cite{Morrey 1956} and Morrey-Eells \cite{MorreyEells}. Further
contributions are in Bolik \cite{Bolik 2001} (for a version in $L^{p}$ and in
H\"{o}lder spaces), Csat\'{o}-Dacorogna-Kneuss \cite{Csato-Dac-Kneuss(livre)},
Iwaniec-Martin \cite{Iwaniec-Martin}, Iwaniec-Scott-Stroffolini
\cite{Iwaniec-Scott-Stroffolini} (for a version in $L^{p}$), Mitrea
\cite{Mitrea 2001}, Mitrea-Mitrea \cite{Mitrea-Mitrea 2002}, Morrey
\cite{Morrey 1966}, Schwarz \cite{Schwarz}, Taylor \cite{Taylor} and von Wahl
\cite{Wahl}. Specifically for the inequality for vector fields in dimension
$2$ and $3,$ we can refer to Amrouche-Bernardi-Dauge-Girault
\cite{Amrouche-Bernardi-Dauge-Girault}, Costabel \cite{Costabel 1991} and
Dautray-Lions \cite{Dautray-Lions}.\smallskip

A stronger version of the classical Gaffney inequality reads as%
\[
\left\Vert \nabla\omega\right\Vert ^{2}\leq C\left(  \left\Vert d\omega
\right\Vert ^{2}+\left\Vert \delta\omega\right\Vert ^{2}+\left\Vert
\omega\right\Vert ^{2}\right)  ,\quad\forall\,\omega\in W_{T}^{d,\delta
,2}\left(  \Omega;\Lambda^{k}\right)
\]
(and similarly with $T$ replaced by $N$) where%
\[
\omega\in W_{T}^{d,\delta,2}\left(  \Omega;\Lambda^{k}\right)  =\left\{
\omega\in L^{2}\left(  \Omega;\Lambda^{k}\right)  :\left\{
\begin{array}
[c]{c}%
d\omega\in L^{2}\left(  \Omega;\Lambda^{k+1}\right) \\
\delta\omega\in L^{2}\left(  \Omega;\Lambda^{k-1}\right) \\
\nu\wedge\omega=0\text{ on }\partial\Omega
\end{array}
\right.  \right\}
\]
and the boundary condition has to be understood in a very weak sense. Clearly
$W_{T}^{1,2}\subset W_{T}^{d,\delta,2}.$ This stronger inequality is, in fact,
a regularity result and is valid for smooth or convex Lipschitz domains
leading, a posteriori, to%
\[
W_{T}^{d,\delta,2}\left(  \Omega;\Lambda^{k}\right)  =W_{T}^{1,2}\left(
\Omega;\Lambda^{k}\right)  .
\]
However, for non-convex Lipschitz domains one has, in general, $W_{T}%
^{1,2}\neq W_{T}^{d,\delta,2}.$ We refer to Mitrea \cite{Mitrea 2001} and
Mitrea-Mitrea \cite{Mitrea-Mitrea 2002}; while for vector fields in dimension
$2$ and $3,$ see Amrouche-Bernardi-Dauge-Girault
\cite{Amrouche-Bernardi-Dauge-Girault}, Ben Belgacem-Bernardi-Costabel-Dauge
\cite{Belgacem-Bernardi-Costabel-Dauge}, Ciarlet-Hazard-Lohrengel
\cite{Ciarlet et alt. 1998}, Costabel \cite{Costabel 1990}, Costabel-Dauge
\cite{Costabel Dauge 1998} and Girault-Raviart \cite{Girault-Raviart}%
.\smallskip

Clearly Gaffney inequality for $k=1$ is reminiscent of Korn inequality. The
best constant in Korn inequality have been investigated by Bauer-Pauly
\cite{Bauer-Pauly} and Desvillettes-Villani \cite{Desvillettes-Villani
(2002)}. Our results (cf. Corollary \ref{Corollaire Gaffney et Korn}) allow us
to recover the best constant found in \cite{Bauer-Pauly}.\smallskip

We should end this introduction with a striking analogy with the classical
Hardy inequality (cf.,\thinspace for example \cite{Marcus-Mizel-Pinchover} and
the bibliography therein). Indeed, classically the best constant $\mu,$ when
the domain is convex (and in fact $\left(  n-1\right)  -$convex, see
\cite{Lewis-Li-Li 2012}), is independent of the dimension (in this case
$\mu=1/4$) and the best constant is not attained; while for general non-convex
domains the best constant is, in general, strictly less than $1/4.$ However
the authors were not able to see if this connection is fortuitous or not.

\section{Notations}

We now fix the notations, for further details we refer to
\cite{Csato-Dac-Kneuss(livre)}.\smallskip

\emph{(i)} A $k-$form $\omega\in\Lambda^{k}=\Lambda^{k}\left(  \mathbb{R}%
^{n}\right)  $ is written as%
\[
\omega=\sum_{1\leq i_{1}<\cdots<i_{k}\leq n}\omega^{i_{1}\cdots i_{k}%
}dx^{i_{1}}\wedge\cdots\wedge dx^{i_{k}}%
\]
when convenient it is identified to a vector in $\mathbb{R}^{\binom{n}{k}}.$
When necessary, we extend, in a natural way, the definition of $\omega
^{i_{1}\cdots i_{k}}$ to any $1\leq i_{1},\cdots,i_{k}\leq n$ (see Notations
2.5 (iii) in \cite{Csato-Dac-Kneuss(livre)}).\smallskip

\emph{(ii)} The exterior product of $\nu\in\Lambda^{1}$ and $\omega\in
\Lambda^{k}$ is $\nu\wedge\omega\in\Lambda^{k+1}$ and is defined as%
\begin{align*}
\nu\wedge\omega &  =\sum_{1\leq i_{1}<\cdots<i_{k}\leq n}\left[  \sum
_{j=1}^{n}\nu^{j}\omega^{i_{1}\cdots i_{k}}\right]  dx^{j}\wedge dx^{i_{1}%
}\wedge\cdots\wedge dx^{i_{k}}\smallskip\\
&  =\sum_{1\leq i_{1}<\cdots<i_{k+1}\leq n}\left[  \sum_{\gamma=1}%
^{k+1}\left(  -1\right)  ^{\gamma-1}\nu^{i_{\gamma}}\omega^{i_{1}\cdots
i_{\gamma-1}i_{\gamma+1}\cdots i_{k+1}}\right]  dx^{i_{1}}\wedge\cdots\wedge
dx^{i_{k+1}}.
\end{align*}

\emph{(iii)} The interior product of $\nu\in\Lambda^{1}$ and $\omega\in
\Lambda^{k}$ is $\nu\,\lrcorner\,\omega\in\Lambda^{k-1}$ and is defined as%
\[
\nu\,\lrcorner\,\omega=\sum_{1\leq i_{1}<\cdots<i_{k-1}\leq n}\left[
\sum_{j=1}^{n}\nu^{j}\omega^{ji_{1}\cdots i_{k-1}}\right]  dx^{i_{1}}%
\wedge\cdots\wedge dx^{i_{k-1}}.
\]

\emph{(iv)} The scalar product of $\omega,\lambda\in\Lambda^{k}$ is defined as%
\[
\left\langle \omega;\lambda\right\rangle =\sum_{1\leq i_{1}<\cdots<i_{k}\leq
n}\left(  \omega^{i_{1}\cdots i_{k}}\lambda^{i_{1}\cdots i_{k}}\right)
=\frac{1}{k!}\,\sum_{1\leq i_{1},\cdots,i_{k}\leq n}\left(  \omega
^{i_{1}\cdots i_{k}}\lambda^{i_{1}\cdots i_{k}}\right)
\]
the associated norm being%
\[
\left\vert \omega\right\vert ^{2}=\sum_{1\leq i_{1}<\cdots<i_{k}\leq n}\left(
\omega^{i_{1}\cdots i_{k}}\right)  ^{2}=\frac{1}{k!}\,\sum_{1\leq i_{1}%
,\cdots,i_{k}\leq n}\left(  \omega^{i_{1}\cdots i_{k}}\right)  ^{2}.
\]
When $k=1$ the interior and the scalar product coincide.\smallskip

\emph{(v)} The Hodge $\ast$ operator associates to $\omega\in\Lambda^{k},$
$\ast\omega\in\Lambda^{n-k}$ via the operation%
\[
\omega\wedge\lambda=\left\langle \ast\omega;\lambda\right\rangle dx^{1}%
\wedge\cdots\wedge dx^{n},\quad\text{for every }\lambda\in\Lambda^{n-k}.
\]
The interior product of $\nu\in\Lambda^{1}$ and $\omega\in\Lambda^{k}$ can be
then written as%
\[
\nu\,\lrcorner\,\omega=\left(  -1\right)  ^{n\left(  k-1\right)  }\ast\left(
\nu\wedge\left(  \ast\omega\right)  \right)  .
\]
We also use several times the identities%
\begin{equation}
\nu\wedge\left(  \nu\,\lrcorner\,\omega\right)  +\nu\,\lrcorner\,\left(
\nu\wedge\omega\right)  =\left\vert \nu\right\vert ^{2}\omega\quad
\text{and}\quad\left\langle \nu\wedge\alpha;\beta\right\rangle =\left\langle
\alpha;\nu\,\lrcorner\,\beta\right\rangle .
\label{equation:decomposing omega by nu}%
\end{equation}

\emph{(vi)} The exterior derivative of $\omega\in\Lambda^{k}$ is $d\omega
\in\Lambda^{k+1}$ and is defined as%
\begin{align*}
d\omega &  =\sum_{1\leq i_{1}<\cdots<i_{k}\leq n}\left[  \sum_{j=1}^{n}%
\omega_{x_{j}}^{i_{1}\cdots i_{k}}\right]  dx^{j}\wedge dx^{i_{1}}\wedge
\cdots\wedge dx^{i_{k}}\smallskip\\
&  =\sum_{1\leq i_{1}<\cdots<i_{k+1}\leq n}\left[  \sum_{\gamma=1}%
^{k+1}\left(  -1\right)  ^{\gamma-1}\omega_{x_{i_{\gamma}}}^{i_{1}\cdots
i_{\gamma-1}i_{\gamma+1}\cdots i_{k+1}}\right]  dx^{i_{1}}\wedge\cdots\wedge
dx^{i_{k+1}}.
\end{align*}
When $k=1$ we can identify $d\omega$ with $\operatorname{curl}\omega
.$\smallskip

\emph{(vii)} The interior derivative (or codifferential) of $\omega\in
\Lambda^{k}$ is $\delta\omega\in\Lambda^{k-1}$ and is defined as%
\[
\delta\omega=\sum_{1\leq i_{1}<\cdots<i_{k-1}\leq n}\left(  \sum_{j=1}%
^{n}\omega_{x_{j}}^{ji_{1}\cdots i_{k-1}}\right)  dx^{i_{1}}\wedge\cdots\wedge
dx^{i_{k-1}}.
\]
When $k=1$ we can identify $\delta\omega$ with $\operatorname{div}\omega
.$\smallskip

\emph{(viii)} The spaces $W_{T}^{1,2}\left(  \Omega;\Lambda^{k}\right)  $ and
$W_{N}^{1,2}\left(  \Omega;\Lambda^{k}\right)  $ are defined as%
\[
W_{T}^{1,2}\left(  \Omega;\Lambda^{k}\right)  =\left\{  \omega\in
W^{1,2}\left(  \Omega;\Lambda^{k}\right)  :\nu\wedge\omega=0\text{ on
}\partial\Omega\right\}
\]%
\[
W_{N}^{1,2}\left(  \Omega;\Lambda^{k}\right)  =\left\{  \omega\in
W^{1,2}\left(  \Omega;\Lambda^{k}\right)  :\nu\,\lrcorner\,\omega=0\text{ on
}\partial\Omega\right\}
\]
where $\nu$ is the outward unit normal to $\partial\Omega.$\smallskip

\emph{(ix)} The sets $\mathcal{H}_{T}\left(  \Omega;\Lambda^{k}\right)  $ and
$\mathcal{H}_{N}\left(  \Omega;\Lambda^{k}\right)  $ are defined as%
\[
\mathcal{H}_{T}\left(  \Omega;\Lambda^{k}\right)  =\left\{  \omega\in
W_{T}^{1,2}\left(  \Omega;\Lambda^{k}\right)  :d\omega=0\text{ and }%
\delta\omega=0\text{ in }\Omega\right\}
\]%
\[
\mathcal{H}_{N}\left(  \Omega;\Lambda^{k}\right)  =\left\{  \omega\in
W_{N}^{1,2}\left(  \Omega;\Lambda^{k}\right)  :d\omega=0\text{ and }%
\delta\omega=0\text{ in }\Omega\right\}  .
\]

\section{Some generalities}

Our first result is the following.

\begin{proposition}
\label{Proposition CT et CN plus grand que 1}Let $\Omega\subset\mathbb{R}^{n}$
be a bounded open set and $0\leq k\leq n.$ Then%
\[
C_{T}\left(  \Omega,k\right)  ,C_{N}\left(  \Omega,k\right)  \geq1.
\]
Moreover%
\[
C_{T}\left(  \Omega,0\right)  =C_{N}\left(  \Omega,0\right)  =C_{T}\left(
\Omega,n\right)  =C_{N}\left(  \Omega,n\right)  =1\quad\text{and}\quad
C_{T}\left(  \Omega,k\right)  =C_{N}\left(  \Omega,n-k\right)  .
\]

\end{proposition}

\begin{remark}
When $k=0$ (respectively $k=n$), $\nabla\omega$ can be identified with
$d\omega$ (respectively $\delta\omega$). Therefore, for any $\omega\in
W^{1,2},$%
\[
\left\Vert \nabla\omega\right\Vert ^{2}=\left\Vert d\omega\right\Vert
^{2}\quad\text{(respectively }\left\Vert \nabla\omega\right\Vert
^{2}=\left\Vert \delta\omega\right\Vert ^{2}\text{).}%
\]
Hence the statements in the proposition when $k=0$ or $n$ are trivial.
\end{remark}

\begin{proof}
\emph{Step 1. }We first prove the main statement. Let $\overline{x}=\left(
\overline{x}_{1},\cdots,\overline{x}_{n}\right)  \in\Omega$ and $0<r<R$ be
such that%
\[
B_{r}\left(  \overline{x}\right)  \subset\Omega\subset B_{R}\left(
\overline{x}\right)  .
\]
Choose a function $\eta\in C_{0}^{\infty}\left(  \Omega\right)  $ such that
$\eta=1$ in $B_{r}$ and $0\leq\eta\leq1$ in $\Omega.$ We extend it to
$\mathbb{R}^{n}$ by $0.$ Define for every $m\in\mathbb{N}$%
\[
\omega_{m}\left(  x\right)  =\sin\left(  mx_{1}\right)  \eta\left(  x\right)
dx^{1}\wedge\cdots\wedge dx^{k}\in C_{0}^{\infty}\left(  \Omega;\Lambda
^{k}\right)  .
\]
Since $\omega_{m}$ vanishes on the boundary of $B_{R}\,,$ we have, by Theorem
5.7 in \cite{Csato-Dac-Kneuss(livre)} (see also \cite{Csato-Dac 2102}),%
\[%
{\displaystyle\int_{B_{R}}}
\left\vert \nabla\omega_{m}\right\vert ^{2}=%
{\displaystyle\int_{B_{R}}}
\left(  \left\vert d\omega_{m}\right\vert ^{2}+\left\vert \delta\omega
_{m}\right\vert ^{2}\right)  \quad\Rightarrow\quad%
{\displaystyle\int_{\Omega}}
\left\vert \nabla\omega_{m}\right\vert ^{2}=%
{\displaystyle\int_{\Omega}}
\left(  \left\vert d\omega_{m}\right\vert ^{2}+\left\vert \delta\omega
_{m}\right\vert ^{2}\right)  .
\]
To prove that $C_{T},C_{N}\geq1$ it is sufficient to show that%
\[
\lim_{m\rightarrow\infty}\frac{\left\Vert d\omega_{m}\right\Vert
^{2}+\left\Vert \delta\omega_{m}\right\Vert ^{2}+\left\Vert \omega
_{m}\right\Vert ^{2}}{\left\Vert \nabla\omega_{m}\right\Vert ^{2}}%
=1+\lim_{m\rightarrow\infty}\frac{\left\Vert \omega_{m}\right\Vert ^{2}%
}{\left\Vert \nabla\omega_{m}\right\Vert ^{2}}=1.
\]
Note that%
\[
\int_{\Omega}\left\vert \omega_{m}\right\vert ^{2}\leq\operatorname*{meas}%
\Omega.
\]
On the other hand we have that%
\[%
{\displaystyle\int_{\Omega}}
\left\vert \nabla\omega_{m}\right\vert ^{2}\geq\int_{B_{r}}\left\vert
\nabla\omega_{m}\right\vert ^{2}=m^{2}\int_{B_{r}}\cos^{2}\left(
mx_{1}\right)  dx.
\]
Since $B_{r}$ is open, there exists $m$ sufficiently large so that%
\[
\left(  \overline{x}_{1},\overline{x}_{1}+\frac{2\,\pi}{m}\right)  \times
B_{r}^{\prime}\subset B_{r}%
\]
where $B_{r}^{\prime}=\left\{  y=\left(  y_{2},\cdots,y_{n}\right)
\in\mathbb{R}^{n-1}:\left\vert y_{i}-\overline{x}_{i}\right\vert \leq\frac
{r}{n}\,,\text{ }i=2,\cdots,n\right\}  $ is independent of $m.$ We thus obtain
that%
\[%
{\displaystyle\int_{\Omega}}
\left\vert \nabla\omega_{m}\right\vert ^{2}\geq m^{2}\operatorname*{meas}%
\left(  B_{r}^{\prime}\right)  \int_{\overline{x}_{1}}^{\overline{x}%
_{1}+2\,\pi/m}\cos^{2}\left(  mx_{1}\right)  dx_{1}\,.
\]
Use now the change of variables $x_{1}=t/m$ to get%
\[%
{\displaystyle\int_{\Omega}}
\left\vert \nabla\omega_{m}\right\vert ^{2}\geq m\operatorname*{meas}\left(
B_{r}^{\prime}\right)  \int_{m\overline{x}_{1}}^{m\overline{x}_{1}+2\,\pi}%
\cos^{2}\left(  t\right)  dt=m\,\pi\operatorname*{meas}\left(  B_{r}^{\prime
}\right)  .
\]
This shows that $\left\Vert \nabla\omega_{m}\right\Vert ^{2}\rightarrow\infty$
and thus $C_{T},C_{N}\geq1$ as asserted.\smallskip

\emph{Step 2. }The fact that $C_{T}\left(  \Omega,k\right)  =C_{N}\left(
\Omega,n-k\right)  $ is immediate through the Hodge $\ast$ operator.\smallskip
\end{proof}

\section{The main theorem}

\subsection{Statement of the theorem}

We now turn to the main theorem that gives several equivalent properties of
$C_{T}\left(  \Omega,k\right)  =1$ (or analogously $C_{N}\left(
\Omega,k\right)  =1$). We start with a definition (cf. for a similar one
\cite{Sha 1986}).

\begin{definition}
\label{Definition k-convexite}Let $\Omega\subset\mathbb{R}^{n}$ be an open
smooth set and $\Sigma=\partial\Omega$ be the associated $\left(  n-1\right)
-$surface. Let $\gamma_{1},\cdots,\gamma_{n-1}$ be the principal curvatures of
$\Sigma.$ Let $1\leq k\leq n-1.$ We say that $\Omega$ is $k-$\emph{convex} if%
\[
\gamma_{i_{1}}+\cdots+\gamma_{i_{k}}\geq0,\quad\text{for every }1\leq
i_{1}<\cdots<i_{k}\leq n-1.
\]

\end{definition}

\begin{remark}
\emph{(i)} When $k=1,$ it is easy to show that $\Omega$ convex implies that
$\Omega$ is $1-$convex. The reverse implication is also true but deeper. The
result is due to Hadamard under slightly stronger conditions and as stated to
Chern-Lashof \cite{Chern-Lashof 1958} (see also Alexander \cite{Alexander
1977}).\smallskip

\emph{(ii)} When $2\leq k\leq n-1,$ the condition that $\Omega$ is $k-$convex
is strictly weaker than saying that $\Omega$ is convex. In particular when
$k=n-1$ the condition means that the mean curvature of $\Sigma=\partial\Omega$
is non-negative.
\end{remark}

We will also use the following quantities.

\begin{definition}
\label{Definition de L et nu}Let $\Omega\subset\mathbb{R}^{n}$ be open and
$\nu\in C^{1}\left(  \overline{\Omega};\Lambda^{1}\right)  .$ We define for
every $0\leq k\leq n$ the two maps%
\[
L^{\nu},K^{\nu}:\Lambda^{k}\left(  \mathbb{R}^{n}\right)  \rightarrow
\Lambda^{k}\left(  \mathbb{R}^{n}\right)
\]
by $L^{\nu}(\omega)=0$ if $k=0$ and $K^{\nu}(\omega)=0$ if $k=n,$ while%
\[
L^{\nu}(\omega)=\sum_{1\leq i_{1}<\cdots<i_{k}\leq n}\omega^{i_{1}\cdots
i_{k}}\,d\left(  \nu\,\lrcorner\,\left(  dx^{i_{1}}\wedge\cdots\wedge
dx^{i_{k}}\right)  \right)  ,\quad\text{if }k\geq1
\]%
\[
K^{\nu}(\omega)=\sum_{1\leq i_{1}<\cdots<i_{k}\leq n}\omega^{i_{1}\cdots
i_{k}}\,\delta\left(  \nu\wedge\left(  dx^{i_{1}}\wedge\cdots\wedge dx^{i_{k}%
}\right)  \right)  ,\quad\text{if }k\leq n-1.
\]

\end{definition}

\begin{remark}
\label{Remarque: permutation de nu avec L et K}\emph{(i) }If $\nu$ is the unit
normal to a surface $\Sigma$ and is extended to a neighborhood of $\Sigma$
such that $\left\vert \nu\right\vert =1$ everywhere, then (see
\cite{Csato-Dac-Kneuss(livre)} Lemma 5.5)%
\[
L^{\nu}\left(  \nu\wedge\alpha\right)  =\nu\wedge L^{\nu}\left(
\alpha\right)  \quad\text{and}\quad K^{\nu}\left(  \nu\,\lrcorner
\,\alpha\right)  =\nu\,\lrcorner\,K^{\nu}\left(  \alpha\right)  .
\]
The right-hand sides of these expressions do not depend on the chosen
extension, see \cite{Csato-Dac-Kneuss(livre)} Theorem 3.23. We will use this
frequently henceforth.\smallskip

\emph{(ii)} In the remaining part of the article we will always assume that
$\nu$ has been extended to a neighborhood of $\Sigma=\partial\Omega$ so as to
have $\left\vert \nu\right\vert =1.$\smallskip

\emph{(iii)} Note that $L^{\nu}$ is linear in $\omega$ and $\nu.$ By
definition it acts pointwise on $\omega.$ In this way, identifying
$\Lambda^{k}\left(  \mathbb{R}^{n}\right)  $ with vectors $\mathbb{R}%
^{\binom{n}{k}},$ the operator $L^{\nu}$ can be seen as a matrix acting on
$\omega.$ But in $\nu,$ on the contrary, $L^{\nu}$ is a local (differential)
operator. The same holds true for $K^{\nu}.$
\end{remark}

We then have the main result.

\begin{theorem}
\label{Theoreme equiv pour CT}Let $\Omega\subset\mathbb{R}^{n}$ be a bounded
open smooth set and $1\leq k\leq n-1.$ Then the following statements are
equivalent.\smallskip

\emph{(i)} $C_{T}\left(  \Omega,k\right)  =1.$\smallskip

\emph{(ii)} For every $\omega\in W_{T}^{1,2}\left(  \Omega;\Lambda^{k}\right)
$%
\[
\widetilde{K}\left(  \omega\right)  =\left\langle K^{\nu}\left(
\nu\,\lrcorner\,\omega\right)  ;\nu\,\lrcorner\,\omega\right\rangle
\geq0,\quad\text{on }\partial\Omega.
\]

\emph{(iii)} The sharper version of Gaffney inequality holds, namely%
\[
\left\Vert \nabla\omega\right\Vert ^{2}\leq\left\Vert d\omega\right\Vert
^{2}+\left\Vert \delta\omega\right\Vert ^{2},\quad\forall\,\omega\in
W_{T}^{1,2}\left(  \Omega;\Lambda^{k}\right)  .
\]

\emph{(iv)} $\Omega$ is $\left(  n-k\right)  -$convex.\smallskip

\emph{(v)} The supremum in (\ref{definition de C T}) is not
attained.\smallskip

\emph{(vi)} $C_{T}$ is scale invariant, namely, for every $t>0$%
\[
C_{T}\left(  t\,\Omega,k\right)  =C_{T}\left(  \Omega,k\right)  .
\]

\end{theorem}

\begin{remark}
\label{Remarque apres Thm equiv pour CT}\emph{(i)} Using the Hodge $\ast$
operation, we obtain immediately, from the theorem, the following equivalent
relations.\smallskip

- $C_{N}\left(  \Omega,k\right)  =1.$\smallskip

- $\widetilde{L}\left(  \omega\right)  =\left\langle L^{\nu}\left(  \nu
\wedge\omega\right)  ;\nu\wedge\omega\right\rangle \geq0,$\quad whenever
$\nu\,\lrcorner\,\omega=0$ on $\partial\Omega.$\smallskip

- The sharper version of Gaffney inequality holds, namely%
\[
\left\Vert \nabla\omega\right\Vert ^{2}\leq\left\Vert d\omega\right\Vert
^{2}+\left\Vert \delta\omega\right\Vert ^{2},\quad\forall\,\omega\in
W_{N}^{1,2}\left(  \Omega;\Lambda^{k}\right)  .
\]

- $\Omega$ is $k-$convex.\smallskip

- The supremum in (\ref{definition de C N}) is not attained.\smallskip

- $C_{N}$ is scale invariant.\smallskip

\emph{(ii)} The condition (ii) of the theorem can be equivalently rewritten
(for any $\omega\in W_{T}^{1,2}\left(  \Omega;\Lambda^{k}\right)  $) as%
\[
\widetilde{K}\left(  \omega\right)  =\left\langle K^{\nu}\left(
\omega\right)  ;\omega\right\rangle \geq0,\quad\text{on }\partial\Omega
\]
since, recalling that $\nu\wedge\omega=0$ (since $\omega\in W_{T}^{1,2}$),%
\begin{align*}
\widetilde{K}\left(  \omega\right)   &  =\left\langle K^{\nu}\left(
\nu\,\lrcorner\,\omega\right)  ;\nu\,\lrcorner\,\omega\right\rangle
=\left\langle \nu\,\lrcorner\,K^{\nu}\left(  \omega\right)  ;\nu
\,\lrcorner\,\omega\right\rangle =\left\langle K^{\nu}\left(  \omega\right)
;\nu\wedge\left(  \nu\,\lrcorner\,\omega\right)  \right\rangle \smallskip\\
&  =\left\langle K^{\nu}\left(  \omega\right)  ;\omega-\nu\,\lrcorner\,\left(
\nu\wedge\omega\right)  \right\rangle =\left\langle K^{\nu}\left(
\omega\right)  ;\omega\right\rangle .
\end{align*}
Similar remarks hold for $\widetilde{L}.$\smallskip

\emph{(iii)} Note that if $C_{T}\left(  \Omega,k\right)  =1,$ then
$\mathcal{H}_{T}\left(  \Omega;\Lambda^{k}\right)  =\left\{  0\right\}  .$
This follows at once from (iii) of the theorem, since non-zero constant forms
cannot satisfy the boundary condition. A similar remark applies to
$\mathcal{H}_{N}\,.$
\end{remark}

\subsection{Some algebraic results}

\begin{lemma}
\label{Lemme 0 algebrique}Let $1\leq k\leq n-1$ and $\lambda_{1}%
,\cdots,\lambda_{k}\in\Lambda^{1}$ with%
\[
\lambda_{i}\,\lrcorner\,\lambda_{j}=0,\quad\text{if }i\neq j.
\]
Then%
\[
\lambda_{i}\,\lrcorner\,\left(  \lambda_{1}\wedge\cdots\wedge\lambda
_{i-1}\wedge\lambda_{i+1}\wedge\cdots\wedge\lambda_{k}\right)  =0,\quad
i=1,\cdots,k
\]%
\[
\left\vert \lambda_{1}\wedge\cdots\wedge\lambda_{k}\right\vert =\left\vert
\lambda_{1}\right\vert \cdots\left\vert \lambda_{k}\right\vert .
\]
Furthermore let%
\[
C^{jl}=\sum_{s_{2},\cdots,s_{k}=1}^{n}\left(  \lambda_{1}\wedge\cdots
\wedge\lambda_{k}\right)  ^{js_{2}\cdots s_{k}}\left(  \lambda_{1}\wedge
\cdots\wedge\lambda_{k}\right)  ^{ls_{2}\cdots s_{k}}.
\]
and $\widehat{\lambda_{j}}=\lambda_{1}\wedge\cdots\lambda_{j-1}\wedge
\lambda_{j+1}\wedge\cdots\wedge\lambda_{k}\,,$ then%
\[
C^{jl}=\left(  \left(  k-1\right)  !\right)  \sum_{\gamma=1}^{k}\left\vert
\widehat{\lambda_{\gamma}}\right\vert ^{2}\lambda_{\gamma}^{j}\lambda_{\gamma
}^{l}\,.
\]

\end{lemma}

\begin{proof}
\emph{Step 1.} We first establish by induction that%
\[
\lambda_{k}\,\lrcorner\,\left(  \lambda_{1}\wedge\cdots\wedge\lambda
_{k-1}\right)  =0
\]
(and similarly for all the other $\lambda_{i}$). Indeed if $k=2,$ this is our
hypothesis; so assume that the result has been proved for $k$ and let us prove
it for $k+1.$ We know from Proposition 2.16 in \cite{Csato-Dac-Kneuss(livre)}
that%
\[
\lambda_{k+1}\,\lrcorner\,\left(  \lambda_{1}\wedge\cdots\wedge\lambda
_{k}\right)  =\left(  \lambda_{k+1}\,\lrcorner\,\lambda_{1}\right)
\wedge\left(  \lambda_{2}\wedge\cdots\wedge\lambda_{k}\right)  -\lambda
_{1}\wedge\left(  \lambda_{k+1}\,\lrcorner\,\left(  \lambda_{2}\wedge
\cdots\wedge\lambda_{k}\right)  \right)
\]
applying the hypothesis of induction we have the result.\smallskip

\emph{Step 2.} We also proceed by induction and show that%
\[
\left\vert \lambda_{1}\wedge\cdots\wedge\lambda_{k}\right\vert =\left\vert
\lambda_{1}\right\vert \cdots\left\vert \lambda_{k}\right\vert .
\]
When $k=2$ we have still from Proposition 2.16 in
\cite{Csato-Dac-Kneuss(livre)} that%
\[
\left\vert \lambda_{1}\wedge\lambda_{2}\right\vert ^{2}=\left\vert \lambda
_{1}\right\vert ^{2}\left\vert \lambda_{2}\right\vert ^{2}-\left\vert
\lambda_{1}\,\lrcorner\,\lambda_{2}\right\vert ^{2}=\left\vert \lambda
_{1}\right\vert ^{2}\left\vert \lambda_{2}\right\vert ^{2}%
\]
as wished. So let us assume that the result has been proved for $k$ and let us
establish it for $k+1.$ By the very same proposition as above we get%
\begin{align*}
\left\vert \lambda_{1}\right\vert ^{2}\left\vert \lambda_{1}\wedge\lambda
_{2}\wedge\cdots\wedge\lambda_{k}\wedge\lambda_{k+1}\right\vert ^{2}  &
=\left\vert \lambda_{1}\wedge\lambda_{1}\wedge\lambda_{2}\wedge\cdots
\wedge\lambda_{k+1}\right\vert ^{2}+\left\vert \lambda_{1}\,\lrcorner\,\left(
\lambda_{1}\wedge\cdots\wedge\lambda_{k+1}\right)  \right\vert ^{2}%
\smallskip\\
&  =\left\vert \lambda_{1}\,\lrcorner\,\left(  \lambda_{1}\wedge\cdots
\wedge\lambda_{k+1}\right)  \right\vert ^{2}.
\end{align*}
Moreover since%
\[
\lambda_{1}\,\lrcorner\,\left(  \lambda_{1}\wedge\cdots\wedge\lambda
_{k+1}\right)  =\left(  \lambda_{1}\,\lrcorner\,\lambda_{1}\right)
\wedge\left(  \lambda_{2}\wedge\cdots\wedge\lambda_{k+1}\right)  -\lambda
_{1}\wedge\left(  \lambda_{1}\,\lrcorner\,\left(  \lambda_{2}\wedge
\cdots\wedge\lambda_{k+1}\right)  \right)
\]
using Step 1 and the hypothesis of induction, we infer that%
\[
\left\vert \lambda_{1}\,\lrcorner\,\left(  \lambda_{1}\wedge\cdots
\wedge\lambda_{k+1}\right)  \right\vert ^{2}=\left\vert \lambda_{1}\right\vert
^{4}\left\vert \lambda_{2}\wedge\cdots\wedge\lambda_{k+1}\right\vert
^{2}=\left\vert \lambda_{1}\right\vert ^{4}\left\vert \lambda_{2}\right\vert
^{2}\cdots\left\vert \lambda_{k+1}\right\vert ^{2}.
\]
Combining the results we have indeed proved our claim.\smallskip

\emph{Step 3.} Writing%
\[
\left(  \lambda_{1}\wedge\cdots\wedge\lambda_{k}\right)  ^{js_{2}\cdots s_{k}%
}=\det\left[  \left(  \lambda_{r}^{s}\right)  _{r=1,\cdots,k}^{s=j,s_{2}%
,\cdots,s_{k}}\right]
\]
we find that%
\begin{align*}
\left(  \lambda_{1}\wedge\cdots\wedge\lambda_{k}\right)  ^{js_{2}\cdots
s_{k}}  &  =\sum_{\gamma=1}^{k}\left(  -1\right)  ^{\gamma+1}\lambda_{\gamma
}^{j}\left(  \lambda_{1}\wedge\cdots\lambda_{\gamma-1}\wedge\lambda_{\gamma
+1}\wedge\cdots\wedge\lambda_{k}\right)  ^{s_{2}\cdots s_{k}}\smallskip\\
&  =\sum_{\gamma=1}^{k}\left(  -1\right)  ^{\gamma+1}\lambda_{\gamma}%
^{j}\left(  \widehat{\lambda_{\gamma}}\right)  ^{s_{2}\cdots s_{k}}%
\end{align*}%
\begin{align*}
\left(  \lambda_{1}\wedge\cdots\wedge\lambda_{k}\right)  ^{ls_{2}\cdots
s_{k}}  &  =\sum_{\delta=1}^{k}\left(  -1\right)  ^{\delta+1}\lambda_{\delta
}^{l}\left(  \lambda_{1}\wedge\cdots\lambda_{\delta-1}\wedge\lambda_{\delta
+1}\wedge\cdots\wedge\lambda_{k}\right)  ^{s_{2}\cdots s_{k}}\smallskip\\
&  =\sum_{\delta=1}^{k}\left(  -1\right)  ^{\delta+1}\lambda_{\delta}%
^{l}\left(  \widehat{\lambda_{\delta}}\right)  ^{s_{2}\cdots s_{k}}.
\end{align*}
Observing (using Steps 1 and 2) that%
\[
\sum_{s_{2},\cdots,s_{k}=1}^{n}\left(  \widehat{\lambda_{\gamma}}\right)
^{s_{2}\cdots s_{k}}\left(  \widehat{\lambda_{\delta}}\right)  ^{s_{2}\cdots
s_{k}}=\left(  \left(  k-1\right)  !\right)  \left\langle \widehat{\lambda
_{\gamma}};\widehat{\lambda_{\delta}}\right\rangle =\left\{
\begin{array}
[c]{cl}%
\left(  k-1\right)  !\,\left\vert \widehat{\lambda_{\gamma}}\right\vert ^{2} &
\text{if }\gamma=\delta\smallskip\\
0 & \text{if }\gamma\neq\delta
\end{array}
\right.
\]
we find that%
\[
C^{jl}=\left(  \left(  k-1\right)  !\right)  \sum_{\gamma=1}^{k}\left\vert
\widehat{\lambda_{\gamma}}\right\vert ^{2}\lambda_{\gamma}^{j}\lambda_{\gamma
}^{l}\,.
\]
as claimed.\smallskip
\end{proof}

We next give a way of computing the quantity $K^{\nu}.$

\begin{lemma}
\label{Lemme 1 algebrique}Let $1\leq k\leq n,$ $\Sigma\subset\mathbb{R}^{n}$
be a smooth $\left(  n-1\right)  -$surface with unit normal $\nu$ and
$\Omega\subset\mathbb{R}^{n}$ be a neighborhood of $\Sigma.$ Let $\alpha
,\beta\in C^{1}\left(  \Omega;\Lambda^{k}\right)  $ be such that, on $\Sigma,$%
\[
\nu\wedge\alpha=\nu\wedge\beta=0.
\]
Then the following equation holds true, on $\Sigma,$%
\begin{equation}
\left\langle K^{\nu}\left(  \nu\,\lrcorner\,\alpha\right)  ;\nu\,\lrcorner
\,\beta\right\rangle +\left\langle K^{\nu}\left(  \nu\,\lrcorner
\,\beta\right)  ;\nu\,\lrcorner\,\alpha\right\rangle =\left\langle
\delta\alpha;\nu\,\lrcorner\,\beta\right\rangle +\left\langle \delta\beta
;\nu\,\lrcorner\,\alpha\right\rangle -\left\langle \nabla\left(
\alpha\,\lrcorner\,\beta\right)  ;\nu\right\rangle .
\label{Equation 1 dans Lemme 1 algebrique}%
\end{equation}
In particular if $\alpha=\nu\wedge\lambda$ in $\Omega$ with%
\[
\lambda=\lambda_{1}\wedge\cdots\wedge\lambda_{k-1}%
\]
where $\lambda_{1},\cdots,\lambda_{k-1}\in C^{1}\left(  \Omega;\Lambda
^{1}\right)  $ with, for every $i,j=1,\cdots,k-1$ and for every $x\in\Omega,$%
\[
\left\vert \nu\left(  x\right)  \right\vert =1,\quad\nu\left(  x\right)
\,\lrcorner\,\lambda_{i}\left(  x\right)  =0\quad\text{and}\quad\lambda
_{i}\left(  x\right)  \,\lrcorner\,\lambda_{j}\left(  x\right)  =\delta
_{ij}\,,
\]
then, in $\Omega,$%
\[
\left\langle K^{\nu}(\lambda);\lambda\right\rangle =\left\langle \delta\left(
\nu\wedge\lambda\right)  ;\lambda\right\rangle .
\]

\end{lemma}

\begin{proof}
\emph{Step 1.} Applying Lemma 5.6 in \cite{Csato-Dac-Kneuss(livre)} we find%
\[
\left\langle K^{\nu}\left(  \nu\,\lrcorner\,\alpha\right)  ;\nu\,\lrcorner
\,\beta\right\rangle =\left\langle \nu\,\lrcorner\,K^{\nu}\left(
\alpha\right)  ;\nu\,\lrcorner\,\beta\right\rangle =\left\langle \delta
\alpha;\nu\,\lrcorner\,\beta\right\rangle -\sum_{I}\left\langle \nabla
\alpha_{I};\nu\right\rangle \beta_{I}%
\]
where $I=\left(  i_{1},\cdots,i_{k}\right)  \in\mathbb{N}^{k}$ with $1\leq
i_{1}<\cdots<i_{k}\leq n.$ We thus deduce that%
\[
\left\langle \nu\,\lrcorner\,K^{\nu}\left(  \alpha\right)  ;\nu\,\lrcorner
\,\beta\right\rangle +\left\langle \nu\,\lrcorner\,K^{\nu}\left(
\beta\right)  ;\nu\,\lrcorner\,\alpha\right\rangle =\left\langle \delta
\alpha;\nu\,\lrcorner\,\beta\right\rangle +\left\langle \delta\beta
;\nu\,\lrcorner\,\alpha\right\rangle -\sum_{I}\left\langle \nabla\left(
\alpha_{I}\beta_{I}\right)  ;\nu\right\rangle
\]
or in other words%
\[
\left\langle \nu\,\lrcorner\,K^{\nu}\left(  \alpha\right)  ;\nu\,\lrcorner
\,\beta\right\rangle +\left\langle \nu\,\lrcorner\,K^{\nu}\left(
\beta\right)  ;\nu\,\lrcorner\,\alpha\right\rangle =\left\langle \delta
\alpha;\nu\,\lrcorner\,\beta\right\rangle +\left\langle \delta\beta
;\nu\,\lrcorner\,\alpha\right\rangle -\left\langle \nabla\left(
\alpha\,\lrcorner\,\beta\right)  ;\nu\right\rangle
\]
which is exactly (\ref{Equation 1 dans Lemme 1 algebrique}).\smallskip

\emph{Step 2.} The extra statement follows from Step 1 applied to
$\beta=\alpha$ and Lemma \ref{Lemme 0 algebrique} since then $\nu
\,\lrcorner\,\alpha=\lambda$ and%
\[
\left\langle \nabla\left(  \alpha\,\lrcorner\,\alpha\right)  ;\nu\right\rangle
=\left\langle \nabla\left(  \left\vert \lambda\right\vert ^{2}\right)
;\nu\right\rangle =0.
\]
The proof is therefore complete.\smallskip
\end{proof}

\subsection{Calculation of sums of principal curvatures}

In the sequel $\Omega\subset\mathbb{R}^{n}$ will always be a bounded open
smooth set with exterior unit normal $\nu.$ When we say that $E_{1}%
,\cdots,E_{n-1}$ is an \emph{orthonormal frame field of principal directions}
of $\partial\Omega$ with \emph{associated principal curvatures} $\gamma
_{1},\cdots,\gamma_{n-1}\,,$ we mean that $\left\{  \nu,E_{1},\cdots
,E_{n-1}\right\}  $ form an orthonormal basis of $\mathbb{R}^{n}$ and, for
every $1\leq i\leq n-1,$%
\begin{equation}
\sum_{j=1}^{n}\left(  E_{i}^{j}\nu_{x_{j}}^{l}\right)  =\gamma_{i}\,E_{i}%
^{l}\quad\Rightarrow\quad\gamma_{i}=\sum_{j,l=1}^{n}\left(  E_{i}^{l}E_{i}%
^{j}\nu_{x_{j}}^{l}\right)  . \label{Equation courbures}%
\end{equation}

\begin{lemma}
\label{Lemme dans CT=1}Let $E_{1},\cdots,E_{n-1}$ be an orthonormal frame
field of principal directions of $\partial\Omega$ with associated principal
curvatures $\gamma_{1},\cdots,\gamma_{n-1}\,.$ Assume that $E_{1}%
,\cdots,E_{n-1}$ are extended locally to a neighborhood of $\partial\Omega$
such that $\left\{  \nu,E_{1},\cdots,E_{n-1}\right\}  $ form an orthonormal
frame field of $\mathbb{R}^{n}.$ Let $1\leq k\leq n-1,$%
\[
\lambda=E_{i_{1}\cdots i_{k-1}}=E_{i_{1}}\wedge\cdots\wedge E_{i_{k-1}}%
\quad\text{and}\quad\mu=E_{j_{1}\cdots j_{k-1}}=E_{j_{1}}\wedge\cdots\wedge
E_{j_{k-1}}%
\]
(for $k=1,$ $\lambda=\mu=1$). Then%
\[
\left\langle \delta\left(  \nu\wedge\lambda\right)  ;\lambda\right\rangle
=\left(  \gamma_{1}+\cdots+\gamma_{n-1}\right)  -\left(  \gamma_{i_{1}}%
+\cdots+\gamma_{i_{k-1}}\right)
\]
while for $\lambda\neq\mu$%
\[
\left\langle \delta\left(  \nu\wedge\lambda\right)  ;\mu\right\rangle
+\left\langle \delta\left(  \nu\wedge\mu\right)  ;\lambda\right\rangle =0.
\]

\end{lemma}

\begin{proof}
The case $k=1$ is immediate. Indeed the first equation reads as%
\[
\left\langle \delta\left(  \nu\wedge\lambda\right)  ;\lambda\right\rangle
=\delta\left(  \nu\right)  =\operatorname{div}\left(  \nu\right)  =\gamma
_{1}+\cdots+\gamma_{n-1}\,.
\]
While nothing is to be proved for the second equation. So we discuss now the
case $k\geq2.$\smallskip

\emph{Step 1:} $k\geq2$\emph{ (first equation).} We now prove that if
$\lambda=E_{i_{1}\cdots i_{k-1}}=E_{i_{1}}\wedge\cdots\wedge E_{i_{k-1}}\,,$
then%
\[
\left\langle \delta\left(  \nu\wedge\lambda\right)  ;\lambda\right\rangle
=\left(  \gamma_{1}+\cdots+\gamma_{n-1}\right)  -\left(  \gamma_{i_{1}}%
+\cdots+\gamma_{i_{k-1}}\right)
\]

\emph{(i)} We find that%
\[
\left\langle \delta\left(  \nu\wedge\lambda\right)  ;\lambda\right\rangle
=\frac{1}{\left(  k-1\right)  !}\sum_{s_{1},\cdots,s_{k-1}=1}^{n}\left[
\sum_{j=1}^{n}\left(  \nu\wedge\lambda\right)  _{x_{j}}^{js_{1}\cdots s_{k-1}%
}\right]  \lambda^{s_{1}\cdots s_{k-1}}=A+B
\]
where (recalling that $\left\vert \lambda\right\vert =1$)%
\begin{align*}
A  &  =\frac{1}{\left(  k-1\right)  !}\sum_{j,s_{1},\cdots,s_{k-1}=1}%
^{n}\left[  \nu^{j}\lambda^{s_{1}\cdots s_{k-1}}\right]  _{x_{j}}%
\lambda^{s_{1}\cdots s_{k-1}}\smallskip\\
&  =\frac{1}{\left(  k-1\right)  !}\sum_{j,s_{1},\cdots,s_{k-1}=1}^{n}%
\nu_{x_{j}}^{j}\left(  \lambda^{s_{1}\cdots s_{k-1}}\right)  ^{2}+\frac
{1}{\left(  k-1\right)  !}\sum_{j=1}^{n}\nu^{j}\sum_{s_{1},\cdots,s_{k-1}%
=1}^{n}\left[  \left(  \frac{\lambda^{s_{1}\cdots s_{k-1}}}{2}\right)
^{2}\right]  _{x_{j}}\smallskip\\
&  =\operatorname{div}\left(  \nu\right)  =\gamma_{1}+\cdots+\gamma_{n-1}%
\end{align*}
and%
\begin{align*}
B  &  =\frac{1}{\left(  k-1\right)  !}\sum_{j,s_{1},\cdots,s_{k-1}=1}%
^{n}\left[  \sum_{r=1}^{k-1}\left(  -1\right)  ^{r}\nu^{s_{r}}\lambda
^{js_{1}\cdots s_{r-1}s_{r+1}\cdots s_{k-1}}\right]  _{x_{j}}\lambda
^{s_{1}\cdots s_{k-1}}\smallskip\\
&  =\frac{1}{\left(  k-1\right)  !}\sum_{j,s_{1},\cdots,s_{k-1}=1}^{n}\left[
\sum_{r=1}^{k-1}\left(  -1\right)  ^{r}\nu_{x_{j}}^{s_{r}}\lambda
^{js_{1}\cdots s_{r-1}s_{r+1}\cdots s_{k-1}}\right]  \lambda^{s_{1}\cdots
s_{k-1}}\smallskip\\
&  +\frac{1}{\left(  k-1\right)  !}\sum_{j,s_{1},\cdots,s_{k-1}=1}^{n}\left[
\sum_{r=1}^{k-1}\left(  -1\right)  ^{r}\nu^{s_{r}}\lambda_{x_{j}}%
^{js_{1}\cdots s_{r-1}s_{r+1}\cdots s_{k-1}}\right]  \lambda^{s_{1}\cdots
s_{k-1}}\smallskip\\
&  =B_{1}+B_{2}\,.
\end{align*}
The result will be established once we prove that%
\begin{align*}
B_{1}  &  =\frac{1}{\left(  k-1\right)  !}\sum_{j,s_{1},\cdots,s_{k-1}=1}%
^{n}\left[  \sum_{r=1}^{k-1}\left(  -1\right)  ^{r}\nu_{x_{j}}^{s_{r}}%
\lambda^{js_{1}\cdots s_{r-1}s_{r+1}\cdots s_{k-1}}\right]  \lambda
^{s_{1}\cdots s_{k-1}}\smallskip\\
&  =-\left(  \gamma_{i_{1}}+\cdots+\gamma_{i_{k-1}}\right)
\end{align*}
and%
\[
B_{2}=\frac{1}{\left(  k-1\right)  !}\sum_{j,s_{1},\cdots,s_{k-1}=1}%
^{n}\left[  \sum_{r=1}^{k-1}\left(  -1\right)  ^{r}\nu^{s_{r}}\lambda_{x_{j}%
}^{js_{1}\cdots s_{r-1}s_{r+1}\cdots s_{k-1}}\right]  \lambda^{s_{1}\cdots
s_{k-1}}=0.
\]

\emph{(ii)} Let us first prove that $B_{2}=0$ (recalling that $\lambda
=E_{i_{1}}\wedge\cdots\wedge E_{i_{k-1}}$). We write%
\[
\left(  k-1\right)  !\,B_{2}=\sum_{r=1}^{k-1}\,\sum_{j,s_{1},\cdots
s_{r-1},s_{r+1},\cdots,s_{k-1}=1}^{n}\left(  \sum_{s_{r}=1}^{n}\left(
-1\right)  ^{r}\nu^{s_{r}}\lambda^{s_{1}\cdots s_{k-1}}\right)  \lambda
_{x_{j}}^{js_{1}\cdots s_{r-1}s_{r+1}\cdots s_{k-1}}.
\]
Observe that, for every $r=1,\cdots,k-1,$%
\[
\sum_{s_{r}=1}^{n}\left(  -1\right)  ^{r}\nu^{s_{r}}\left(  E_{i_{1}}%
\wedge\cdots\wedge E_{i_{k-1}}\right)  ^{s_{1}\cdots s_{k-1}}=-\left(
\nu\,\lrcorner\,\left(  E_{i_{1}}\wedge\cdots\wedge E_{i_{k-1}}\right)
\right)  ^{s_{1}\cdots s_{r-1}s_{r+1}\cdots s_{k-1}}=0,
\]
in view of Lemma \ref{Lemme 0 algebrique}, leading to the fact that $B_{2}%
=0.$\smallskip

\emph{(iii)} We finally show that $B_{1}=-\left(  \gamma_{i_{1}}+\cdots
+\gamma_{i_{k-1}}\right)  .$ Note that, interchanging the positions of the
indices $s_{r}$ and $s_{r^{\prime}}\,,$ we get (recalling that $\lambda
=E_{i_{1}}\wedge\cdots\wedge E_{i_{k-1}}$)%
\begin{align*}
B_{1}  &  =\frac{1}{\left(  k-1\right)  !}\sum_{j,s_{1},\cdots,s_{k-1}=1}%
^{n}\left[  \sum_{r=1}^{k-1}\left(  -1\right)  ^{r}\nu_{x_{j}}^{s_{r}}%
\lambda^{js_{1}\cdots s_{r-1}s_{r+1}\cdots s_{k-1}}\right]  \lambda
^{s_{1}\cdots s_{k-1}}\smallskip\\
&  =\frac{-1}{\left(  k-2\right)  !}\sum_{j,s_{1},\cdots,s_{k-1}=1}^{n}\left[
\nu_{x_{j}}^{s_{1}}\left(  E_{i_{1}}\wedge\cdots\wedge E_{i_{k-1}}\right)
^{js_{2}\cdots s_{k-1}}\left(  E_{i_{1}}\wedge\cdots\wedge E_{i_{k-1}}\right)
^{s_{1}\cdots s_{k-1}}\right]  .
\end{align*}
The result follows (cf.\thinspace Lemma \ref{Lemme 0 algebrique}) since, for
every $1\leq j,s_{1}\leq n,$%
\[
\sum_{r=1}^{k-1}\left(  E_{i_{r}}^{j}E_{i_{r}}^{s_{1}}\right)  =\frac
{1}{\left(  k-2\right)  !}\sum_{s_{2},\cdots,s_{k-1}=1}^{n}\left(  E_{i_{1}%
}\wedge\cdots\wedge E_{i_{k-1}}\right)  ^{js_{2}\cdots s_{k-1}}\left(
E_{i_{1}}\wedge\cdots\wedge E_{i_{k-1}}\right)  ^{s_{1}\cdots s_{k-1}}%
\]
and thus%
\[
B_{1}=-\sum_{r=1}^{k-1}\left(  \sum_{j,s_{1}=1}^{n}\nu_{x_{j}}^{s_{1}}%
E_{i_{r}}^{j}E_{i_{r}}^{s_{1}}\right)  =-\left(  \gamma_{i_{1}}+\cdots
+\gamma_{i_{k-1}}\right)
\]
which is exactly what had to be proved.\smallskip

\emph{Step 2:} $k\geq2$\emph{ (second equation).} We finally establish that if%
\[
\lambda=E_{i_{1}\cdots i_{k-1}}=E_{i_{1}}\wedge\cdots\wedge E_{i_{k-1}}%
\quad\text{and}\quad\mu=E_{j_{1}\cdots j_{k-1}}=E_{j_{1}}\wedge\cdots\wedge
E_{j_{k-1}}\,.
\]
and $\lambda\neq\mu,$ then%
\[
\left\langle \delta\left(  \nu\wedge\lambda\right)  ;\mu\right\rangle
+\left\langle \delta\left(  \nu\wedge\mu\right)  ;\lambda\right\rangle =0.
\]
This amounts to showing that $X=0$ where%
\begin{equation}
X=\sum_{s_{1},\cdots,s_{k-1}=1}^{n}\left[  \sum_{j=1}^{n}\left(  \nu
\wedge\lambda\right)  _{x_{j}}^{js_{1}\cdots s_{k-1}}\right]  \mu^{s_{1}\cdots
s_{k-1}}+\sum_{s_{1},\cdots,s_{k-1}=1}^{n}\left[  \sum_{j=1}^{n}\left(
\nu\wedge\mu\right)  _{x_{j}}^{js_{1}\cdots s_{k-1}}\right]  \lambda
^{s_{1}\cdots s_{k-1}}. \label{Equation 6 dans CT=1}%
\end{equation}
We write $X=A+B$ where%
\begin{align*}
A  &  =\sum_{j,s_{1},\cdots,s_{k-1}=1}^{n}\left[  \nu^{j}\lambda^{s_{1}\cdots
s_{k-1}}\right]  _{x_{j}}\mu^{s_{1}\cdots s_{k-1}}+\sum_{j,s_{1}%
,\cdots,s_{k-1}=1}^{n}\left[  \nu^{j}\mu^{s_{1}\cdots s_{k-1}}\right]
_{x_{j}}\lambda^{s_{1}\cdots s_{k-1}}\smallskip\\
&  =2\sum_{j}^{n}\nu_{x_{j}}^{j}\sum_{s_{1},\cdots,s_{k-1}=1}^{n}%
\lambda^{s_{1}\cdots s_{k-1}}\mu^{s_{1}\cdots s_{k-1}}+\sum_{j=1}^{n}\nu
^{j}\sum_{s_{1},\cdots,s_{k-1}=1}^{n}\left[  \lambda^{s_{1}\cdots s_{k-1}}%
\mu^{s_{1}\cdots s_{k-1}}\right]  _{x_{j}}=0
\end{align*}
(since $\left\langle \lambda;\mu\right\rangle =0$ by Lemma
\ref{Lemme 0 algebrique}) and%
\begin{align*}
B  &  =\sum_{j,s_{1},\cdots,s_{k-1}=1}^{n}\left[  \sum_{r=1}^{k-1}\left(
-1\right)  ^{r}\nu^{s_{r}}\lambda^{js_{1}\cdots s_{r-1}s_{r+1}\cdots s_{k-1}%
}\right]  _{x_{j}}\mu^{s_{1}\cdots s_{k-1}}\smallskip\\
&  +\sum_{j,s_{1},\cdots,s_{k-1}=1}^{n}\left[  \sum_{r=1}^{k-1}\left(
-1\right)  ^{r}\nu^{s_{r}}\mu^{js_{1}\cdots s_{r-1}s_{r+1}\cdots s_{k-1}%
}\right]  _{x_{j}}\lambda^{s_{1}\cdots s_{k-1}}%
\end{align*}
which leads to $B=B_{1}+B_{2}$ where%
\begin{align*}
B_{1}  &  =\sum_{j,s_{1},\cdots,s_{k-1}=1}^{n}\left[  \sum_{r=1}^{k-1}\left(
-1\right)  ^{r}\nu_{x_{j}}^{s_{r}}\lambda^{js_{1}\cdots s_{r-1}s_{r+1}\cdots
s_{k-1}}\right]  \mu^{s_{1}\cdots s_{k-1}}\smallskip\\
&  +\sum_{j,s_{1},\cdots,s_{k-1}=1}^{n}\left[  \sum_{r=1}^{k-1}\left(
-1\right)  ^{r}\nu_{x_{j}}^{s_{r}}\mu^{js_{1}\cdots s_{r-1}s_{r+1}\cdots
s_{k-1}}\right]  \lambda^{s_{1}\cdots s_{k-1}}%
\end{align*}%
\begin{align*}
B_{2}  &  =\sum_{j,s_{1},\cdots,s_{k-1}=1}^{n}\left[  \sum_{r=1}^{k-1}\left(
-1\right)  ^{r}\nu^{s_{r}}\lambda_{x_{j}}^{js_{1}\cdots s_{r-1}s_{r+1}\cdots
s_{k-1}}\right]  \mu^{s_{1}\cdots s_{k-1}}\smallskip\\
&  +\sum_{j,s_{1},\cdots,s_{k-1}=1}^{n}\left[  \sum_{r=1}^{k-1}\left(
-1\right)  ^{r}\nu^{s_{r}}\mu_{x_{j}}^{js_{1}\cdots s_{r-1}s_{r+1}\cdots
s_{k-1}}\right]  \lambda^{s_{1}\cdots s_{k-1}}.
\end{align*}
It remains to prove, in order to show that $X=0$ where $X$ is as in
(\ref{Equation 6 dans CT=1}), that $B_{1}=B_{2}=0.$\smallskip

\emph{(i)} We start with the fact that $B_{2}=0.$ We rewrite the definition as%
\begin{align*}
B_{2}  &  =\sum_{r=1}^{k-1}\,\sum_{j,s_{1},\cdots s_{r-1},s_{r+1}%
,\cdots,s_{k-1}=1}^{n}\left[  \sum_{s_{r}=1}^{n}\left(  -1\right)  ^{r}%
\nu^{s_{r}}\mu^{s_{1}\cdots s_{k-1}}\right]  \lambda_{x_{j}}^{js_{1}\cdots
s_{r-1}s_{r+1}\cdots s_{k-1}}\smallskip\\
&  +\sum_{r=1}^{k-1}\,\sum_{j,s_{1},\cdots s_{r-1},s_{r+1},\cdots,s_{k-1}%
=1}^{n}\left[  \sum_{s_{r}=1}^{n}\left(  -1\right)  ^{r}\nu^{s_{r}}%
\lambda^{s_{1}\cdots s_{k-1}}\right]  \mu_{x_{j}}^{js_{1}\cdots s_{r-1}%
s_{r+1}\cdots s_{k-1}}.
\end{align*}
Since, for every $r=1,\cdots,k-1,$%
\[
\sum_{s_{r}=1}^{n}\left(  -1\right)  ^{r}\nu^{s_{r}}\lambda^{s_{1}\cdots
s_{k-1}}=-\left(  \nu\,\lrcorner\,\lambda\right)  ^{s_{1}\cdots s_{r-1}%
s_{r+1}\cdots s_{k-1}}=0
\]%
\[
\sum_{s_{r}=1}^{n}\left(  -1\right)  ^{r}\nu^{s_{r}}\mu^{s_{1}\cdots s_{k-1}%
}=-\left(  \nu\,\lrcorner\,\mu\right)  ^{s_{1}\cdots s_{r-1}s_{r+1}\cdots
s_{k-1}}=0,
\]
we find that indeed $B_{2}=0.$\smallskip

\emph{(ii)} We finally prove that $B_{1}=0.$ Note that, interchanging the
positions of the indices $s_{r}$ and $s_{r^{\prime}}\,,$ we obtain%
\[
B_{1}=-\left(  k-1\right)  \sum_{j,s_{1},\cdots,s_{k-1}=1}^{n}\nu_{x_{j}%
}^{s_{1}}\left[  \lambda^{js_{2}\cdots s_{k-1}}\mu^{s_{1}\cdots s_{k-1}}%
+\mu^{js_{2}\cdots s_{k-1}}\lambda^{s_{1}\cdots s_{k-1}}\right]  .
\]
The result follows if we can show that, if $\lambda\neq\mu$ where%
\[
\lambda=\lambda_{1}\wedge\cdots\wedge\lambda_{k-1}\quad\text{and}\quad\mu
=\mu_{1}\wedge\cdots\wedge\mu_{k-1}\,,
\]
with $\lambda_{i},\mu_{i}\in\left\{  E_{1},\cdots,E_{n-1}\right\}  ,$ then%
\[
\sum_{j,s_{1},\cdots,s_{k-1}=1}^{n}\nu_{x_{j}}^{s_{1}}\left[  \lambda
^{js_{2}\cdots s_{k-1}}\mu^{s_{1}\cdots s_{k-1}}\right]  =\sum_{j,s_{1}%
,\cdots,s_{k-1}=1}^{n}\nu_{x_{j}}^{s_{1}}\left[  \mu^{js_{2}\cdots s_{k-1}%
}\lambda^{s_{1}\cdots s_{k-1}}\right]  =0.
\]
Since both identities are established similarly, we prove only the first one,
namely%
\begin{equation}
\sum_{j,s_{1},\cdots,s_{k-1}=1}^{n}\nu_{x_{j}}^{s_{1}}\left[  \left(
\lambda_{1}\wedge\cdots\wedge\lambda_{k-1}\right)  ^{js_{2}\cdots s_{k-1}%
}\left(  \mu_{1}\wedge\cdots\wedge\mu_{k-1}\right)  ^{s_{1}s_{2}\cdots
s_{k-1}}\right]  =0. \label{Equation 7 dans CT=1}%
\end{equation}
Since $\lambda\neq\mu$ we assume, up to reordering, that $\lambda_{1}\neq
\mu_{1}\,.$ We claim that%
\begin{equation}
\left.
\begin{array}
[c]{ccl}%
C^{js_{1}} & = &
{\displaystyle\sum\limits_{s_{2},\cdots,s_{k-1}=1}^{n}}
\left[  \left(  \lambda_{1}\wedge\cdots\wedge\lambda_{k-1}\right)
^{js_{2}\cdots s_{k-1}}\left(  \mu_{1}\wedge\cdots\wedge\mu_{k-1}\right)
^{s_{1}s_{2}\cdots s_{k-1}}\right]  \smallskip\\
& = & \left(  \left(  k-2\right)  !\right)  \,\lambda_{1}^{j}\mu_{1}^{s_{1}%
}\,\left\langle \lambda_{2}\wedge\cdots\wedge\lambda_{k-1};\mu_{2}\wedge
\cdots\wedge\mu_{k-1}\right\rangle .
\end{array}
\right.  \label{Equation 8 dans CT=1}%
\end{equation}
which leads to%
\begin{align*}
&  \sum_{j,s_{1},\cdots,s_{k-1}=1}^{n}\nu_{x_{j}}^{s_{1}}\left[  \left(
\lambda_{1}\wedge\cdots\wedge\lambda_{k-1}\right)  ^{js_{2}\cdots s_{k-1}%
}\left(  \mu_{1}\wedge\cdots\wedge\mu_{k-1}\right)  ^{s_{1}s_{2}\cdots
s_{k-1}}\right]  \smallskip\\
&  =\left(  \left(  k-2\right)  !\right)  \left\langle \lambda_{2}\wedge
\cdots\wedge\lambda_{k-1};\mu_{2}\wedge\cdots\wedge\mu_{k-1}\right\rangle
\sum_{j,s_{1}=1}^{n}\nu_{x_{j}}^{s_{1}}\lambda_{1}^{j}\mu_{1}^{s_{1}%
}\smallskip\\
&  =\left(  \left(  k-2\right)  !\right)  \left\langle \lambda_{2}\wedge
\cdots\wedge\lambda_{k-1};\mu_{2}\wedge\cdots\wedge\mu_{k-1}\right\rangle
\gamma_{\lambda_{1}}\,\sum_{s_{1}=1}^{n}\lambda_{1}^{s_{1}}\mu_{1}^{s_{1}}=0
\end{align*}
(where $\gamma_{\lambda_{1}}$ is the principal curvature corresponding to
$\lambda_{1}$) which is exactly (\ref{Equation 7 dans CT=1}). It remains to
show (\ref{Equation 8 dans CT=1}). We have%
\begin{align*}
\left(  \lambda_{1}\wedge\cdots\wedge\lambda_{k-1}\right)  ^{js_{2}\cdots
s_{k-1}}  &  =\sum_{r=1}^{k-1}\left(  -1\right)  ^{r+1}\lambda_{r}^{j}\left(
\widehat{\lambda_{r}}\right)  ^{s_{2}\cdots s_{k-1}}\smallskip\\
\left(  \mu_{1}\wedge\cdots\wedge\mu_{k-1}\right)  ^{s_{1}s_{2}\cdots
s_{k-1}}  &  =\sum_{t=1}^{k-1}\left(  -1\right)  ^{t+1}\mu_{t}^{s_{1}}\left(
\widehat{\mu_{t}}\right)  ^{s_{2}\cdots s_{k-1}}.
\end{align*}
where%
\[
\widehat{\lambda_{r}}=\lambda_{1}\wedge\cdots\lambda_{r-1}\wedge\lambda
_{r+1}\wedge\cdots\wedge\lambda_{k-1}\quad\text{and}\quad\widehat{\mu_{t}}%
=\mu_{1}\wedge\cdots\mu_{t-1}\wedge\mu_{t+1}\wedge\cdots\wedge\mu_{k-1}\,.
\]
We therefore have that%
\begin{align*}
C^{js_{1}}  &  =\sum_{r,t=1}^{k-1}\left(  -1\right)  ^{r+t}\lambda_{r}^{j}%
\mu_{t}^{s_{1}}\sum_{s_{2},\cdots,s_{k-1}=1}^{n}\left(  \widehat{\lambda_{r}%
}\right)  ^{s_{2}\cdots s_{k-1}}\left(  \widehat{\mu_{t}}\right)
^{s_{2}\cdots s_{k-1}}\smallskip\\
&  =\left(  \left(  k-2\right)  !\right)  \sum_{r,t=1}^{k-1}\left(  -1\right)
^{r+t}\lambda_{r}^{j}\mu_{t}^{s_{1}}\left\langle \widehat{\lambda_{r}%
};\widehat{\mu_{t}}\right\rangle .
\end{align*}
Invoking Lemma \ref{Lemme 0 algebrique} and the fact that $\lambda_{1}\neq
\mu_{1}\,,$ we obtain that, unless $r=t=1,$%
\[
\left\langle \widehat{\lambda_{r}};\widehat{\mu_{t}}\right\rangle =0
\]
leading to (\ref{Equation 8 dans CT=1}). The proof is therefore complete.
\end{proof}

\subsection{Formulas for $L^{\nu}$ and $K^{\nu}$ in terms of principal
curvatures}

We first prove the symmetry of $L^{\nu}$ and $K^{\nu},$ which essentially
follows from the symmetry of the second fundamental form of a hypersurface. We
use Remark \ref{Remarque: permutation de nu avec L et K} (i)--(iii) in the
following lemma and its proof. In particular, recall that we have extended
$\nu$ in a neighborhood of $\Sigma$ such that for any $\alpha,\beta$
\[
\left[  L^{\nu}\left(  \nu\wedge\alpha\right)  =\nu\wedge L^{\nu}\left(
\alpha\right)  \quad\text{and}\quad K^{\nu}\left(  \nu\,\lrcorner
\,\alpha\right)  =\nu\,\lrcorner\,K^{\nu}\left(  \alpha\right)  \right]
\quad\text{on }\Sigma.
\]

\begin{lemma}
\label{lemma:L nu and K nu symmetric}Let $\Sigma\subset\mathbb{R}^{n}$ be a
smooth $n-1$ dimensional hypersurface with unit normal $\nu$ and let $1\leq
k\leq n-1.$ Then at every point $x_{0}$ of $\Sigma$ and for every
$\alpha,\,\beta\in\Lambda^{k}\left(  \mathbb{R}^{n}\right)  $ the following
two identities hold%
\[
\left\langle L^{\nu}\left(  \nu\wedge\alpha\right)  ;\nu\wedge\beta
\right\rangle =\left\langle L^{\nu}\left(  \nu\wedge\beta\right)  ;\nu
\wedge\alpha\right\rangle
\]%
\[
\left\langle K^{\nu}\left(  \nu\,\lrcorner\,\alpha\right)  ;\nu\,\lrcorner
\,\beta\right\rangle =\left\langle K^{\nu}\left(  \nu\,\lrcorner
\,\beta\right)  ;\nu\,\lrcorner\,\alpha\right\rangle .
\]

\end{lemma}

\begin{proof}
We only prove the first one, the second one follows by duality
(\cite{Csato-Dac-Kneuss(livre)} Lemma 5.3).\smallskip

\emph{Step 1.} Note that since $\alpha\in\Lambda^{k}\left(  \mathbb{R}%
^{n}\right)  ,$ i.e. has constant coefficients, $L^{\nu}\left(  \alpha\right)
=d\left(  \nu\,\lrcorner\,\alpha\right)  .$ Let us prove that we can assume%
\[
\nu\left(  x_{0}\right)  =\left(  1,0,\cdots,0\right)  =e_{1}\,.
\]
Choose $A\in O(n)$ such that $A^{\ast}(\nu(x_{0}))=e_{1}$ and set $\mu
=A^{\ast}(\nu).$ It has the property that%
\[
\mu\left(  A^{-1}x_{0}\right)  =e_{1}\,.
\]
Set%
\[
\tilde{\alpha}=A^{\ast}\left(  \alpha\right)  \quad\text{and}\quad\tilde
{\beta}=A^{\ast}\left(  \beta\right)  .
\]
It now follows from Theorem 3.10 (note that $A^{\ast}=A^{\sharp}$ if $A\in
O\left(  n\right)  $) and Proposition 2.19 in \cite{Csato-Dac-Kneuss(livre)}
that%
\[
\left\langle \nu\wedge L^{\nu}\left(  \alpha\right)  ;\nu\wedge\beta
\right\rangle =\left\langle A^{\ast}(\nu)\wedge d\left(  A^{\sharp}%
(\nu)\,\lrcorner\,A^{\ast}\left(  \alpha\right)  \right)  ;A^{\ast}(\nu)\wedge
A^{\ast}\left(  \beta\right)  \right\rangle =\left\langle \mu\wedge d\left(
\mu\,\lrcorner\,\tilde{\alpha}\right)  ;\mu\wedge\tilde{\beta}\right\rangle .
\]
Since also $\tilde{\alpha}\in\Lambda^{k}\left(  \mathbb{R}^{n}\right)  ,$ i.e.
has constant coefficients, $d\left(  \mu\,\lrcorner\,\tilde{\alpha}\right)
=L^{\mu}\left(  \tilde{\alpha}\right)  .$ This proves the claim of Step
1.\smallskip

\emph{Step 2.} We now assume that $\nu(x_{0})=e_{1}\,.$ By linearity it is
sufficient to show the claim for%
\[
\alpha=dx^{i_{1}}\wedge\cdots\wedge dx^{i_{k}}\quad\text{and}\quad
\beta=dx^{j_{1}}\wedge\cdots\wedge dx^{j_{k}}.
\]
We distinguish two cases.\smallskip

\emph{Case 1: }$i_{1}=1$\emph{ or }$j_{1}=1.$ We consider only the case
$i_{1}=1,$ the other one being handled similarly. Then $\nu\wedge\alpha=0$
(which implies $L^{\nu}\left(  \nu\wedge\alpha\right)  =0$) and therefore%
\[
\left\langle L^{\nu}\left(  \nu\wedge\alpha\right)  ;\nu\wedge\beta
\right\rangle =0=\left\langle L^{\nu}\left(  \nu\wedge\beta\right)  ;\nu
\wedge\alpha\right\rangle
\]
and the symmetry is proved.\smallskip

\emph{Case 2: }$i_{1}\,,j_{1}>1.$ We have to show that%
\[
\left\langle \nu\wedge d\left(  \nu\,\lrcorner\,\alpha\right)  ;\nu\wedge
\beta\right\rangle =\left\langle \nu\wedge d\left(  \nu\,\lrcorner
\,\beta\right)  ;\nu\wedge\alpha\right\rangle .
\]
Since $j_{1}>1,$ we get at $x_{0}$ that $\beta=e_{1}\,\lrcorner\,\left(
e_{1}\wedge\beta\right)  $ and the same for $\alpha.$ So we have to show that%
\begin{equation}
\left\langle d\left(  \nu\,\lrcorner\,\alpha\right)  ;\beta\right\rangle
=\left\langle d\left(  \nu\,\lrcorner\,\beta\right)  ;\alpha\right\rangle .
\label{eq:sym of L nu simplified}%
\end{equation}
Clearly we can assume that $\left(  i_{1},\cdots,i_{k}\right)  \neq\left(
j_{1},\cdots,j_{k}\right)  .$ By a direct calculation one obtains%
\[
d\left(  \nu\,\lrcorner\,\alpha\right)  =\sum_{s=1}^{n}\sum_{\gamma=1}%
^{k}\left(  -1\right)  ^{\gamma-1}\nu_{x_{s}}^{i_{\gamma}}\,dx^{s}\wedge
dx^{i_{1}}\wedge\cdots\wedge\widehat{dx^{i_{\gamma}}}\wedge\cdots\wedge
dx^{i_{k}},
\]
where $\widehat{a}$ means that $a$ has been omitted. Two possibilities may
then happen.\smallskip

\emph{Case 2.1.} $\left\{  i_{1},\cdots,i_{k}\right\}  $ and $\left\{
j_{1},\cdots,j_{k}\right\}  $ differ in more than one index (considered as
sets). Then for any $s=1,\cdots,n$%
\[
\left\langle dx^{s}\wedge dx^{i_{1}}\wedge\cdots\wedge\widehat{dx^{i_{\gamma}%
}}\wedge\cdots\wedge dx^{i_{k}};dx^{j_{1}}\wedge\cdots\wedge dx^{j_{k}%
}\right\rangle =0.
\]
It follows that $\left\langle d\left(  \nu\,\lrcorner\,\alpha\right)
;\beta\right\rangle =0$ and by symmetry (\ref{eq:sym of L nu simplified})
follows.\smallskip

\emph{Case 2.2.} There exist $\tau,\gamma\in\left\{  1,\cdots,k\right\}  $
such that $\left(  i_{1},\cdots,\widehat{i_{\gamma}},\cdots,i_{k}\right)
=\left(  j_{1},\cdots,\widehat{j_{\tau}},\cdots,j_{k}\right)  .$ Without loss
of generality $\tau\leq\gamma$ and hence%
\[
\left(  j_{1},\cdots,j_{k}\right)  =\left(  i_{1},\cdots,i_{\tau-1},j_{\tau
},i_{\tau},\cdots,\widehat{i_{\gamma}},\cdots,i_{k}\right)  .
\]
We immediately find%
\[
\left\langle d\left(  \nu\,\lrcorner\,\alpha\right)  ;\beta\right\rangle
=\left(  -1\right)  ^{\tau+\gamma}\nu_{x_{j_{\tau}}}^{i_{\gamma}}\,.
\]
In the same way, using that%
\[
\left(  i_{1},\cdots,i_{k}\right)  =\left(  j_{1},\cdots,\widehat{j_{\tau}%
},\cdots,j_{\gamma},i_{\gamma},j_{\gamma+1},\cdots,j_{k}\right)  ,
\]
one obtains%
\[
\left\langle d\left(  \nu\,\lrcorner\,\beta\right)  ;\alpha\right\rangle
=\left(  -1\right)  ^{\tau+\gamma}\nu_{x_{i_{\gamma}}}^{j_{\tau}}\,.
\]
So to prove the symmetry we have to show that%
\[
\nu_{x_{j_{\tau}}}^{i_{\gamma}}=\nu_{x_{i_{\gamma}}}^{j_{\tau}}\quad\text{or
equivalently}\quad\left\langle \nabla\nu\cdot e_{j_{\tau}};e_{i_{\gamma}%
}\right\rangle =\left\langle \nabla\nu\cdot e_{i_{\gamma}};e_{j_{\tau}%
}\right\rangle .
\]
This last equality follows from the symmetry of the second fundamental form of
the hypersurface $\Sigma$ at the point $x_{0},$ because $e_{i_{\gamma}}$ and
$e_{j_{\tau}}$ are tangent vectors.\smallskip
\end{proof}

We now improve Lemma \ref{lemma:L nu and K nu symmetric} (for a different
proof of the next result see Lemma
\ref{lemma:S_k in terms of principal directions}).

\begin{lemma}
\label{lemma:L nu equal bilinear form with curvatures and dual version}Let
$\Sigma\subset\mathbb{R}^{n}$ be a smooth $n-1$ dimensional hypersurface with
unit normal $\nu$ and let $1\leq k\leq n-1.$ Let $E_{1},\cdots,E_{n-1}$ be an
orthonormal set of principal directions of $\Sigma$ with associated principal
curvatures $\gamma_{1},\cdots,\gamma_{n-1}\,.$ Then, at every point $x_{0}%
\in\Sigma$ and for every $\alpha,\,\beta\in\Lambda^{k}\left(  \mathbb{R}%
^{n}\right)  ,$ the following two identities hold%
\begin{equation}
\left\langle L^{\nu}\left(  \nu\wedge\alpha\right)  ;\nu\wedge\beta
\right\rangle =\sum_{1\leq i_{1}<\cdots<i_{k}\leq n-1}\left\langle
\alpha;E_{i_{1}\cdots i_{k}}\right\rangle \left\langle \beta;E_{i_{1}\cdots
i_{k}}\right\rangle \sum_{j\in\{i_{1},\cdots,i_{k}\}}\gamma_{j}
\label{eq:L nu in terms of gamma i}%
\end{equation}%
\begin{equation}
\left\langle K^{\nu}\left(  \nu\,\lrcorner\,\alpha\right)  ;\nu\,\lrcorner
\,\beta\right\rangle =\sum_{1\leq i_{1}<\cdots<i_{k-1}\leq n-1}\left\langle
\alpha;\nu\wedge E_{i_{1}\cdots i_{k-1}}\right\rangle \left\langle \beta
;\nu\wedge E_{i_{1}\cdots i_{k-1}}\right\rangle \sum_{j\notin\left\{
i_{1},\cdots,i_{k-1}\right\}  }\gamma_{j} \label{eq:K nu in terms of gamma i}%
\end{equation}
where $E_{i_{1}\cdots i_{k}}=E_{i_{1}}\wedge\cdots\wedge E_{i_{k}}\,.$
\end{lemma}

\begin{remark}
When $k=n-1,$ the first formula reads as%
\[
\left\langle L^{\nu}\left(  \nu\wedge\alpha\right)  ;\nu\wedge\beta
\right\rangle =\left\langle \alpha;E_{1}\wedge\cdots\wedge E_{n-1}%
\right\rangle \left\langle \beta;E_{1}\wedge\cdots\wedge E_{n-1}\right\rangle
\sum_{j=1}^{n-1}\gamma_{j}%
\]
while, when $k=1,$ the second one reads as%
\[
\left\langle K^{\nu}\left(  \nu\,\lrcorner\,\alpha\right)  ;\nu\,\lrcorner
\,\beta\right\rangle =\left\langle \alpha;\nu\right\rangle \left\langle
\beta;\nu\right\rangle \sum_{j=1}^{n-1}\gamma_{j}\,.
\]

\end{remark}

\begin{proof}
We only prove the second statement. The first one can be deduced from the
other one by duality, using for instance \cite{Csato-Dac-Kneuss(livre)} Lemma
5.3. Since both sides of the equation are bilinear in $\left(  \alpha
,\beta\right)  $ it is sufficient to show the identity for basis vectors of
$\Lambda^{k}\left(  \mathbb{R}^{n}\right)  .$ We choose basis vectors of the
type%
\[
\text{(a)}\quad\alpha=E_{i_{1}}\wedge\cdots\wedge E_{i_{k}}\,,\quad
\beta=E_{j_{1}}\wedge\cdots\wedge E_{j_{k}}%
\]
or of the type%
\[
\text{(b)}\quad\alpha=\nu\wedge E_{i_{1}}\wedge\cdots\wedge E_{i_{k-1}%
}\,,\quad\beta=\nu\wedge E_{j_{1}}\wedge\cdots\wedge E_{j_{k-1}}\,.
\]
If either one of $\alpha$ or $\beta$ is of the type (a), then one immediately
obtains that both sides of (\ref{eq:K nu in terms of gamma i}) are zero and
the equation is trivially satisfied (see Lemma \ref{Lemme 0 algebrique} and
(\ref{equation:decomposing omega by nu})). So we only need to consider the
case that $\alpha$ and $\beta$ are both of type (b). We distinguish two cases.
We also let $x_{0}\in\Sigma.$\smallskip

\emph{Case 1:} $\left(  i_{1},\cdots,i_{k-1}\right)  =\left(  j_{1}%
,\cdots,j_{k-1}\right)  .$ In that case the right hand side of
(\ref{eq:K nu in terms of gamma i}) is equal to%
\[
\sum_{1\leq l_{1}<\cdots<l_{k-1}\leq n-1}\left\langle \alpha;\nu\wedge
E_{l_{1}\cdots l_{k-1}}\right\rangle \left\langle \beta;\nu\wedge
E_{l_{1}\cdots l_{k-1}}\right\rangle \sum_{j\notin\left\{  l_{1}%
,\cdots,l_{k-1}\right\}  }\gamma_{j}=\sum_{j\notin\left\{  i_{1}%
,\cdots,i_{k-1}\right\}  }\gamma_{j}\,.
\]
We now use Lemma \ref{Lemme 1 algebrique} with $\lambda=E_{i_{1}}\wedge
\cdots\wedge E_{i_{k-1}}.$ We can assume that $\left\{  \nu,E_{1}%
,\ldots,E_{n-1}\right\}  $ are extended to an orthonormal basis in a
neighborhood of $x_{0}\,.$ Note that from
(\ref{equation:decomposing omega by nu}) and Lemma \ref{Lemme 0 algebrique} we
get that $\nu\,\lrcorner\,\left(  \nu\wedge\lambda\right)  =\lambda-\nu
\wedge\left(  \nu\,\lrcorner\,\lambda\right)  =\lambda.$ So Lemma
\ref{Lemme 1 algebrique} gives%
\begin{equation}
\left\langle K^{\nu}\left(  \nu\,\lrcorner\,\alpha\right)  ;\nu\,\lrcorner
\,\alpha\right\rangle =\left\langle K^{\nu}\left(  \lambda\right)
;\lambda\right\rangle =\left\langle \delta\left(  \nu\wedge\lambda\right)
;\lambda\right\rangle . \label{Equation 4bis dans CT=1}%
\end{equation}
We get the result appealing to Lemma \ref{Lemme dans CT=1} namely%
\[
\left\langle \delta\left(  \nu\wedge\lambda\right)  ;\lambda\right\rangle
=\left(  \gamma_{1}+\cdots+\gamma_{n-1}\right)  -\left(  \gamma_{i_{1}}%
+\cdots+\gamma_{i_{k-1}}\right)  =\sum_{j\notin\left\{  i_{1},\cdots
,i_{k-1}\right\}  }\gamma_{j}\,.
\]

\emph{Case 2:} $(i_{1},\cdots,i_{k-1})\neq(j_{1},\cdots,j_{k-1}).$ The right
hand side of (\ref{eq:K nu in terms of gamma i}) is now $0.$ So we have to
show that%
\[
\left\langle K^{\nu}\left(  \nu\,\lrcorner\,\alpha\right)  ;\nu\,\lrcorner
\,\beta\right\rangle =0
\]
Let $\lambda=E_{i_{1}}\wedge\cdots\wedge E_{i_{k-1}}$ and $\mu=E_{j_{1}}%
\wedge\cdots\wedge E_{j_{k-1}}.$ It follows from Lemma
\ref{Lemme 1 algebrique} (using (\ref{equation:decomposing omega by nu}) as in
Case 1) that%
\[
\left\langle K^{\nu}\left(  \nu\,\lrcorner\,\alpha\right)  ;\nu\,\lrcorner
\,\beta\right\rangle +\left\langle K^{\nu}\left(  \nu\,\lrcorner
\,\beta\right)  ;\nu\,\lrcorner\,\alpha\right\rangle =\left\langle
\delta\left(  \nu\wedge\lambda\right)  ;\mu\right\rangle +\left\langle
\delta\left(  \nu\wedge\mu\right)  ;\lambda\right\rangle -\left\langle
\nabla\left(  \alpha\,\lrcorner\,\beta\right)  ;\nu\right\rangle .
\]
Recall that $\left\{  \nu,E_{1},\ldots,E_{n-1}\right\}  $ are extended to an
orthonormal basis in a neighborhood of $x_{0}$ and therefore $\nabla\left(
\alpha\,\lrcorner\,\beta\right)  =0.$ Thus it follows from Lemmas
\ref{Lemme dans CT=1} and \ref{lemma:L nu and K nu symmetric} that%
\[
\left\langle K^{\nu}\left(  \nu\,\lrcorner\,\alpha\right)  ;\nu\,\lrcorner
\,\beta\right\rangle =\frac{1}{2}\left(  \left\langle \delta\left(  \nu
\wedge\lambda\right)  ;\mu\right\rangle +\left\langle \delta\left(  \nu
\wedge\mu\right)  ;\lambda\right\rangle \right)  =0,
\]
which proves the claim of the present case.\smallskip
\end{proof}

\subsection{Proof of the main theorem}

For the equivalence (i) $\Leftrightarrow$ (iv), we give below a proof which is
elementary and self-contained. A second proof can be obtained from Theorem
\ref{piecewisesmoothformula} and the remark following it. Still another proof,
more in the language of differential geometry, can be given using Theorem
\ref{thm:fund identity with curvatures on manifold}. These two other proofs
are independent of the one given below and of the previous
analysis..\smallskip

\begin{proof}
(Theorem \ref{Theoreme equiv pour CT}). We know from Theorem 5.7 in
\cite{Csato-Dac-Kneuss(livre)} (see also \cite{Csato-Dac 2102} or Theorem
\ref{thm:fund identity with curvatures} for a slightly different way of
expressing the identity) that, for every $\omega\in W_{T}^{1,2}\left(
\Omega;\Lambda^{k}\right)  \cup W_{N}^{1,2}\left(  \Omega;\Lambda^{k}\right)
,$%
\begin{equation}%
{\displaystyle\int_{\Omega}}
\left(  \left\vert d\omega\right\vert ^{2}+\left\vert \delta\omega\right\vert
^{2}-\left\vert \nabla\omega\right\vert ^{2}\right)  =%
{\displaystyle\int_{\partial\Omega}}
\left(  \left\langle L^{\nu}\left(  \nu\wedge\omega\right)  ;\nu\wedge
\omega\right\rangle +\left\langle K^{\nu}\left(  \nu\,\lrcorner\,\omega
\right)  ;\nu\,\lrcorner\,\omega\right\rangle \right)
.\label{Equation 1 dans CT=1}%
\end{equation}

\emph{Step 1:} \emph{(i) }$\Rightarrow$ \emph{(ii)}. Assume that $C_{T}\left(
\Omega,k\right)  =1.$ This means that, for every $\omega\in W_{T}^{1,2}\left(
\Omega;\Lambda^{k}\right)  ,$%
\[
0\leq\left\Vert d\omega\right\Vert ^{2}+\left\Vert \delta\omega\right\Vert
^{2}-\left\Vert \nabla\omega\right\Vert ^{2}+\left\Vert \omega\right\Vert
^{2}=\left\Vert \omega\right\Vert ^{2}+K\left(  \omega\right)
\]
where%
\[
K\left(  \omega\right)  =%
{\displaystyle\int_{\partial\Omega}}
\widetilde{K}\left(  \omega\right)  =%
{\displaystyle\int_{\partial\Omega}}
\left\langle K^{\nu}\left(  \nu\,\lrcorner\,\omega\right)  ;\nu\,\lrcorner
\,\omega\right\rangle .
\]
Next let $\varphi\in W_{T}^{1,2}\left(  \Omega;\Lambda^{k}\right)  $ be such
that $\varphi=\omega$ on $\partial\Omega$ and $\varphi\equiv0$ in $\Omega$
outside an $\epsilon-$neighborhood of $\partial\Omega.$ Note that, since
$\varphi=\omega$ on $\partial\Omega,$ then%
\[
K^{\nu}\left(  \nu\,\lrcorner\,\varphi\right)  =K^{\nu}\left(  \nu
\,\lrcorner\,\omega\right)  \quad\text{on }\partial\Omega.
\]
We thus have, by (\ref{Equation 1 dans CT=1}) and since $\varphi=\omega$ on
$\partial\Omega,$%
\[
0\leq\left\Vert \varphi\right\Vert ^{2}+K\left(  \varphi\right)  =\left\Vert
\varphi\right\Vert ^{2}+K\left(  \omega\right)  .
\]
Since $\left\Vert \varphi\right\Vert ^{2}$ is as small as we want, we deduce
that%
\begin{equation}
K\left(  \omega\right)  =%
{\displaystyle\int_{\partial\Omega}}
\widetilde{K}\left(  \omega\right)  =%
{\displaystyle\int_{\partial\Omega}}
\left\langle K^{\nu}\left(  \nu\,\lrcorner\,\omega\right)  ;\nu\,\lrcorner
\,\omega\right\rangle \geq0,\quad\forall\,\omega\in W_{T}^{1,2}\left(
\Omega;\Lambda^{k}\right)  .\label{Equation 0 dans CT=1}%
\end{equation}
We now prove (ii) from the above inequality. Choose $\omega\in W_{T}%
^{1,2}\left(  \Omega;\Lambda^{k}\right)  ,$ $\psi\in C^{\infty}\left(
\overline{\Omega}\right)  $ and $\alpha=\psi\,\omega\in W_{T}^{1,2}\left(
\Omega;\Lambda^{k}\right)  .$ Invoking (\ref{Equation 0 dans CT=1}) we find
that%
\[
0\leq K\left(  \alpha\right)  =%
{\displaystyle\int_{\partial\Omega}}
\widetilde{K}\left(  \alpha\right)  =%
{\displaystyle\int_{\partial\Omega}}
\psi^{2}\,\widetilde{K}\left(  \omega\right)  .
\]
Since $\psi$ is arbitrary, we have the claim, i.e. $\widetilde{K}\left(
\omega\right)  \geq0.$\smallskip

\emph{Step 2:} \emph{(ii) }$\Rightarrow$ \emph{(iii)}. From
(\ref{Equation 1 dans CT=1}) we have%
\[
\left\Vert d\omega\right\Vert ^{2}+\left\Vert \delta\omega\right\Vert
^{2}-\left\Vert \nabla\omega\right\Vert ^{2}=K\left(  \omega\right)  =%
{\displaystyle\int_{\partial\Omega}}
\widetilde{K}\left(  \omega\right)
\]
and thus the result, since $K\left(  \omega\right)  \geq0$ (because
$\widetilde{K}\left(  \omega\right)  \geq0$).\smallskip

\emph{Step 3:} \emph{(iii) }$\Rightarrow$ \emph{(i)}. This is trivial, once
coupled with Proposition \ref{Proposition CT et CN plus grand que 1}%
.\smallskip

\emph{Step 4:} \emph{(ii) }$\Rightarrow$ \emph{(iv)}. We choose in (ii), for
$1\leq i_{1}<\cdots<i_{k-1}\leq n-1,$%
\[
\omega=\nu\wedge\lambda\quad\text{with}\quad\lambda=E_{i_{1}}\wedge
\cdots\wedge E_{i_{k-1}}%
\]
From the assumption and the second conclusion in Lemma
\ref{lemma:L nu equal bilinear form with curvatures and dual version}, we then
obtain%
\[
0\leq\left\langle K^{\nu}\left(  \nu\,\lrcorner\,\omega\right)  ;\nu
\,\lrcorner\,\omega\right\rangle =\sum_{j\notin\left\{  i_{1},\cdots
,i_{k-1}\right\}  }\gamma_{j}\,.
\]

\emph{Step 5:} \emph{(iv) }$\Rightarrow$ \emph{(ii)}. This follows from the
second conclusion of Lemma
\ref{lemma:L nu equal bilinear form with curvatures and dual version}%
.\smallskip

\emph{Step 6:} \emph{(iii) }$\Rightarrow$ \emph{(v)}. The fact that the
supremum is not attained follows from (iii), since, for every $\omega\in
W_{T}^{1,2}\left(  \Omega;\Lambda^{k}\right)  ,$%
\[
\left\Vert \nabla\omega\right\Vert ^{2}\leq\left\Vert d\omega\right\Vert
^{2}+\left\Vert \delta\omega\right\Vert ^{2}\leq\left\Vert d\omega\right\Vert
^{2}+\left\Vert \delta\omega\right\Vert ^{2}+\left\Vert \omega\right\Vert ^{2}%
\]
hence the result.\smallskip

\emph{Step 7:} \emph{(v) }$\Rightarrow$ \emph{(i).} In order to prove the
statement, we show that if $C_{T}>1,$ then there exists a maximizer. We divide
the proof into three substeps.\smallskip

\emph{Step 7.1.} Let $\omega_{s}\in W_{T}^{1,2}\left(  \Omega;\Lambda
^{k}\right)  \setminus\left\{  0\right\}  $ be a maximizing sequence (and
hence $\omega_{s}$ is not a constant form), i.e.%
\[
\lim_{s\rightarrow\infty}\frac{\left\Vert \nabla\omega_{s}\right\Vert ^{2}%
}{\left\Vert d\omega_{s}\right\Vert ^{2}+\left\Vert \delta\omega
_{s}\right\Vert ^{2}+\left\Vert \omega_{s}\right\Vert ^{2}}=C_{T}\,.
\]
Without loss of generality, up to replacing $\omega_{s}$ by $\omega
_{s}/\left\Vert \nabla\omega_{s}\right\Vert ,$ we can assume that $\left\Vert
\nabla\omega_{s}\right\Vert =1$ and hence%
\begin{equation}
\lim_{s\rightarrow\infty}\left[  \left\Vert d\omega_{s}\right\Vert
^{2}+\left\Vert \delta\omega_{s}\right\Vert ^{2}+\left\Vert \omega
_{s}\right\Vert ^{2}\right]  =\frac{1}{C_{T}}<1.
\label{Equation1 dans maximiseur}%
\end{equation}
In particular $\left\Vert \omega_{s}\right\Vert $ is bounded and thus, up to a
subsequence that we do not relabel, there exists $\omega\in W_{T}^{1,2}\left(
\Omega;\Lambda^{k}\right)  $ such that%
\[
\omega_{s}\rightharpoonup\omega\quad\text{in }W^{1,2}.
\]
We prove in the next substeps that $\omega$ is a maximizer.\smallskip

\emph{Step 7.2.} We first show that $\omega\neq0.$ Suppose, for the sake of
contradiction, that $\omega=0,$ then from (\ref{Equation1 dans maximiseur}) we
get%
\begin{equation}
\lim_{s\rightarrow\infty}\left[  \left\Vert d\omega_{s}\right\Vert
^{2}+\left\Vert \delta\omega_{s}\right\Vert ^{2}\right]  =\frac{1}{C_{T}}<1.
\label{Equation2 dans maximiseur}%
\end{equation}
From (\ref{Equation 1 dans CT=1}), we infer that there exists $c_{1}%
=c_{1}\left(  \Omega\right)  $ such that%
\[
\left\Vert d\omega_{s}\right\Vert ^{2}+\left\Vert \delta\omega_{s}\right\Vert
^{2}=1+%
{\displaystyle\int_{\partial\Omega}}
\left\langle K^{\nu}\left(  \nu\,\lrcorner\,\omega_{s}\right)  ;\nu
\,\lrcorner\,\omega_{s}\right\rangle \geq1-c_{1}%
{\displaystyle\int_{\partial\Omega}}
\left\vert \omega_{s}\right\vert ^{2}.
\]
Since (cf. Proposition 5.15 in \cite{Csato-Dac-Kneuss(livre)}) there exists
$c_{2}=c_{2}\left(  \Omega\right)  $ such that for every $\epsilon>0$%
\[%
{\displaystyle\int_{\partial\Omega}}
\left\vert \omega_{s}\right\vert ^{2}\leq\epsilon\left\Vert \nabla\omega
_{s}\right\Vert ^{2}+\frac{c_{2}}{\epsilon}\left\Vert \omega_{s}\right\Vert
^{2}=\epsilon+\frac{c_{2}}{\epsilon}\left\Vert \omega_{s}\right\Vert ^{2}%
\]
we deduce that%
\[
\left\Vert d\omega_{s}\right\Vert ^{2}+\left\Vert \delta\omega_{s}\right\Vert
^{2}\geq1-c_{1}\epsilon-\frac{c_{1}c_{2}}{\epsilon}\left\Vert \omega
_{s}\right\Vert ^{2}.
\]
Letting $s\rightarrow\infty$ we find%
\[
\lim_{s\rightarrow\infty}\left[  \left\Vert d\omega_{s}\right\Vert
^{2}+\left\Vert \delta\omega_{s}\right\Vert ^{2}\right]  \geq1-c_{1}\epsilon
\]
and, since $\epsilon$ is arbitrary, we find a contradiction with
(\ref{Equation2 dans maximiseur}).\smallskip

\emph{Step 7.3.} We may now conclude. In the sequel we will have to pass
several times to subsequences in order that all limits are true limits but,
for the sake of not burdening the notations, we do not relabel these
subsequences.\smallskip

\emph{(i)} We have, recalling that $\omega\in W_{T}^{1,2}\left(
\Omega;\Lambda^{k}\right)  ,$%
\begin{align*}
1  &  =\left\Vert \nabla\omega_{s}\right\Vert ^{2}=\left\Vert \nabla\left(
\omega_{s}-\omega\right)  \right\Vert ^{2}+\left\Vert \nabla\omega\right\Vert
^{2}+2\int_{\Omega}\left\langle \nabla\left(  \omega_{s}-\omega\right)
;\nabla\omega\right\rangle \smallskip\\
&  \leq\left\Vert \nabla\left(  \omega_{s}-\omega\right)  \right\Vert
^{2}+C_{T}\left(  \left\Vert d\omega\right\Vert _{L^{2}}^{2}+\left\Vert
\delta\omega\right\Vert _{L^{2}}^{2}+\left\Vert \omega\right\Vert _{L^{2}}%
^{2}\right)  +2\int_{\Omega}\left\langle \nabla\left(  \omega_{s}%
-\omega\right)  ;\nabla\omega\right\rangle .
\end{align*}
Since $\omega_{s}\rightharpoonup\omega$ in $W^{1,2},$ we find that%
\[
1=\lim_{s\rightarrow\infty}\left\Vert \nabla\left(  \omega_{s}-\omega\right)
\right\Vert ^{2}+\left\Vert \nabla\omega\right\Vert ^{2}\leq\lim
_{s\rightarrow\infty}\left\Vert \nabla\left(  \omega_{s}-\omega\right)
\right\Vert ^{2}+C_{T}\left(  \left\Vert d\omega\right\Vert _{L^{2}}%
^{2}+\left\Vert \delta\omega\right\Vert _{L^{2}}^{2}+\left\Vert \omega
\right\Vert _{L^{2}}^{2}\right)  .
\]

\emph{(ii)} Since $\omega_{s}-\omega\in W_{T}^{1,2}\left(  \Omega;\Lambda
^{k}\right)  ,$ we have%
\begin{align*}
\left\Vert \nabla\left(  \omega_{s}-\omega\right)  \right\Vert ^{2}  &  \leq
C_{T}\left(  \left\Vert d\left(  \omega_{s}-\omega\right)  \right\Vert
_{L^{2}}^{2}+\left\Vert \delta\left(  \omega_{s}-\omega\right)  \right\Vert
_{L^{2}}^{2}+\left\Vert \left(  \omega_{s}-\omega\right)  \right\Vert _{L^{2}%
}^{2}\right)  \smallskip\\
&  =C_{T}\left(  \left\Vert d\omega_{s}\right\Vert ^{2}+\left\Vert
\delta\omega_{s}\right\Vert ^{2}+\left\Vert \omega_{s}\right\Vert
^{2}+\left\Vert d\omega\right\Vert _{L^{2}}^{2}+\left\Vert \delta
\omega\right\Vert _{L^{2}}^{2}+\left\Vert \omega\right\Vert _{L^{2}}%
^{2}\right)  \smallskip\\
&  -C_{T}\left(  \int_{\Omega}2\left[  \left\langle d\omega_{s};d\omega
\right\rangle +\left\langle \delta\omega_{s};\delta\omega\right\rangle
+\left\langle \omega_{s};\omega\right\rangle \right]  \right)  .
\end{align*}
Passing to the limit, recalling (\ref{Equation1 dans maximiseur}) and that
$\omega_{s}\rightharpoonup\omega$ in $W^{1,2},$ we get%
\[
\lim_{s\rightarrow\infty}\left\Vert \nabla\left(  \omega_{s}-\omega\right)
\right\Vert ^{2}\leq1-C_{T}\left(  \left\Vert d\omega\right\Vert _{L^{2}}%
^{2}+\left\Vert \delta\omega\right\Vert _{L^{2}}^{2}+\left\Vert \omega
\right\Vert _{L^{2}}^{2}\right)  .
\]

\emph{(iii)} Combining (i) and (ii) we obtain%
\begin{align*}
1  &  =\lim_{s\rightarrow\infty}\left\Vert \nabla\left(  \omega_{s}%
-\omega\right)  \right\Vert ^{2}+\left\Vert \nabla\omega\right\Vert
^{2}\smallskip\\
&  \leq\lim_{s\rightarrow\infty}\left\Vert \nabla\left(  \omega_{s}%
-\omega\right)  \right\Vert ^{2}+C_{T}\left(  \left\Vert d\omega\right\Vert
_{L^{2}}^{2}+\left\Vert \delta\omega\right\Vert _{L^{2}}^{2}+\left\Vert
\omega\right\Vert _{L^{2}}^{2}\right)  \leq1
\end{align*}
which implies that%
\[
\left\Vert \nabla\omega\right\Vert ^{2}=C_{T}\left(  \left\Vert d\omega
\right\Vert _{L^{2}}^{2}+\left\Vert \delta\omega\right\Vert _{L^{2}}%
^{2}+\left\Vert \omega\right\Vert _{L^{2}}^{2}\right)
\]
as wished.\smallskip

\emph{Step 8:} \emph{(i) }$\Rightarrow$ \emph{(vi).} Since (i) (and thus
(iii)) holds, we find, for $\omega\in W_{T}^{1,2}\left(  t\,\Omega;\Lambda
^{k}\right)  $ and setting $\omega\left(  x\right)  =u\left(  x/t\right)  ,$%
\begin{align*}
\left\Vert \nabla\omega\right\Vert _{L^{2}\left(  t\,\Omega\right)  }^{2}  &
=\int_{t\,\Omega}\left\vert \nabla\omega\left(  y\right)  \right\vert
^{2}dy=t^{n-2}\int_{\Omega}\left\vert \nabla u\left(  x\right)  \right\vert
^{2}dx=t^{n-2}\left\Vert \nabla u\right\Vert _{L^{2}\left(  \Omega\right)
}^{2}\smallskip\\
&  \leq t^{n-2}\left\Vert du\right\Vert _{L^{2}\left(  \Omega\right)  }%
^{2}+t^{n-2}\left\Vert \delta u\right\Vert _{L^{2}\left(  \Omega\right)  }%
^{2}=\left\Vert d\omega\right\Vert _{L^{2}\left(  t\,\Omega\right)  }%
^{2}+\left\Vert \delta\omega\right\Vert _{L^{2}\left(  t\,\Omega\right)  }^{2}%
\end{align*}
which shows that $C_{T}\left(  t\,\Omega,k\right)  =C_{T}\left(
\Omega,k\right)  =1.$\smallskip

\emph{Step 9:} \emph{(vi) }$\Rightarrow$ \emph{(i).} Without loss of
generality we can assume that $t<1.$ We reason by contradiction and assume
that $C_{T}\left(  \Omega,k\right)  >1.$ Invoking (v) we have that there
exists $u\in W_{T}^{1,2}\left(  \Omega;\Lambda^{k}\right)  $ such that%
\[
C_{T}\left(  \Omega,k\right)  =\frac{\left\Vert \nabla u\right\Vert
_{L^{2}\left(  \Omega\right)  }^{2}}{\left\Vert du\right\Vert _{L^{2}\left(
\Omega\right)  }^{2}+\left\Vert \delta u\right\Vert _{L^{2}\left(
\Omega\right)  }^{2}+\left\Vert u\right\Vert _{L^{2}\left(  \Omega\right)
}^{2}}\,.
\]
Setting $\omega\left(  x\right)  =u\left(  x/t\right)  ,$ we obtain that
$\omega\in W_{T}^{1,2}\left(  t\,\Omega;\Lambda^{k}\right)  $ and%
\[
C_{T}\left(  \Omega,k\right)  =\frac{\left\Vert \nabla\omega\right\Vert
_{L^{2}\left(  t\,\Omega\right)  }^{2}}{\left\Vert d\omega\right\Vert
_{L^{2}\left(  t\,\Omega\right)  }^{2}+\left\Vert \delta\omega\right\Vert
_{L^{2}\left(  t\,\Omega\right)  }^{2}+t^{-2}\left\Vert \omega\right\Vert
_{L^{2}\left(  t\,\Omega\right)  }^{2}}\,.
\]
Since $t<1,$ we get%
\[
C_{T}\left(  \Omega,k\right)  <\frac{\left\Vert \nabla\omega\right\Vert
_{L^{2}\left(  t\,\Omega\right)  }^{2}}{\left\Vert d\omega\right\Vert
_{L^{2}\left(  t\,\Omega\right)  }^{2}+\left\Vert \delta\omega\right\Vert
_{L^{2}\left(  t\,\Omega\right)  }^{2}+\left\Vert \omega\right\Vert
_{L^{2}\left(  t\,\Omega\right)  }^{2}}\leq C_{T}\left(  t\,\Omega,k\right)
\]
which is our claim.\smallskip
\end{proof}

Theorem \ref{Theoreme equiv pour CT} (combined with Remark
\ref{Remarque apres Thm equiv pour CT}) has as an immediate corollary the following.

\begin{corollary}
Let $\Omega\subset\mathbb{R}^{n}$ be a bounded open smooth set and $k=1.$
Then\smallskip

\emph{(i)} $C_{T}\left(  \Omega,1\right)  =1$ if and only if the mean
curvature of $\partial\Omega$ is non-negative;\smallskip

\emph{(ii)} $C_{N}\left(  \Omega,1\right)  =1$ if and only if $\Omega$ is convex.
\end{corollary}

\section{Some examples}

We now deal with some special cases where we can make $C_{T},C_{N}$
arbitrarily large.

\begin{proposition}
\label{Proposition C arbitr grand}Let $1\leq k\leq n-1.$ Then there exists a
set $\Omega_{k}\subset B$ (a fixed ball of $\mathbb{R}^{n}$) such that
$C_{T}\left(  \Omega_{k},k\right)  ,C_{N}\left(  \Omega_{k},n-k\right)  $ are
arbitrarily large.
\end{proposition}

\begin{remark}
Except for the case $k=1,$ the sets $\Omega_{k}$ that we construct are not
smooth. However it is easy to modify slightly these sets so as to make them
smooth, while preserving the proposition.
\end{remark}

\begin{proof}
Since $C_{T}\left(  \Omega,k\right)  =C_{N}\left(  \Omega,n-k\right)  ,$ it is
sufficient to prove the result for $C_{T}\left(  \Omega,k\right)  .$ For the
sake of clarity we deal with the case $k=1$ separately.\smallskip

\emph{Step 1 (}$k=1$\emph{).} We let, for $x\in\mathbb{R}^{n},$ $\left\vert
x\right\vert $ denote the usual Euclidean norm. Let $0<r<1$ and $\Omega
_{1}=\left\{  x\in\mathbb{R}^{n}:r<\left\vert x\right\vert <1\right\}  .$ We
then choose $\lambda\in C^{1}\left(  \left[  r,1\right]  \right)  $ arbitrary
and%
\[
\omega\left(  x\right)  =\lambda\left(  \left\vert x\right\vert \right)
\sum_{i=1}^{n}x_{i}\,dx^{i}\in W_{T}^{1,2}\left(  \Omega_{1};\Lambda
^{1}\right)  .
\]
Clearly%
\[
\omega_{x_{i}}^{j}=\lambda\,\delta^{ij}+\lambda^{\prime}\,\frac{x_{i}x_{j}%
}{\left\vert x\right\vert }\,,\quad d\omega=0\quad\text{and}\quad\delta
\omega=\operatorname{div}\left(  x\,\lambda\right)  =n\,\lambda+\left\vert
x\right\vert \lambda^{\prime}%
\]
leading to%
\[
\left\vert \delta\omega\right\vert ^{2}=\left(  n\,\lambda+\left\vert
x\right\vert \lambda^{\prime}\right)  ^{2}\quad\text{and}\quad\left\vert
\nabla\omega\right\vert ^{2}=n\,\lambda^{2}+2\left\vert x\right\vert
\lambda\,\lambda^{\prime}+\left\vert x\right\vert ^{2}\left(  \lambda^{\prime
}\right)  ^{2}.
\]
Choose $\lambda\left(  s\right)  =s^{-n}$ (with this choice we have
$\delta\omega=0$). We therefore have (denoting by $\sigma_{n}$ the measure of
the unit sphere of $\mathbb{R}^{n}$) that%
\[%
{\displaystyle\int_{\Omega_{1}}}
\left\vert \nabla\omega\right\vert ^{2}=\sigma_{n}\int_{r}^{1}\left(
n^{2}-n\right)  s^{-n-1}ds=\left.  \frac{\sigma_{n}\left(  n^{2}-n\right)
s^{-n}}{-n}\right\vert _{r}^{1}=\sigma_{n}\left(  n-1\right)  \left[
r^{-n}-1\right]
\]
while%
\[%
{\displaystyle\int_{\Omega_{1}}}
\left(  \left\vert d\omega\right\vert ^{2}+\left\vert \delta\omega\right\vert
^{2}+\left\vert \omega\right\vert ^{2}\right)  =\sigma_{n}\int_{r}^{1}%
s^{-n+1}ds=\left\{
\begin{array}
[c]{cl}%
\frac{\sigma_{n}}{n-2}\left[  r^{-n+2}-1\right]  & \text{if }n>2\smallskip\\
-\sigma_{n}\log r & \text{if }n=2.
\end{array}
\right.
\]
Therefore, when $r\rightarrow0,$ we find (writing $\sim$ for the asymptotic
behavior)%
\[
\frac{\left\Vert \nabla\omega\right\Vert ^{2}}{\left\Vert d\omega\right\Vert
^{2}+\left\Vert \delta\omega\right\Vert ^{2}+\left\Vert \omega\right\Vert
^{2}}\sim\left\{
\begin{array}
[c]{cl}%
\frac{\left(  n-2\right)  \left(  n-1\right)  }{r^{2}} & \text{if
}n>2\smallskip\\
-\frac{1}{r^{2}\log r} & \text{if }n=2.
\end{array}
\right.
\]
Thus, for $r$ sufficiently small, we deduce that $C_{T}\left(  \Omega
_{1},1\right)  $ is arbitrarily large as wished.\smallskip

\emph{Step 2 (}$2\leq k\leq n-1$\emph{).} We divide the proof into two
parts.\smallskip

\emph{Step 2.1.}\textbf{ }Let us introduce some notations.\smallskip

1) We write for $x=\left(  x_{1},\cdots,x_{n}\right)  \in\mathbb{R}^{n}$%
\[
\left\vert x\right\vert _{k}=\sqrt{x_{1}^{2}+\cdots+x_{n-k+1}^{2}}\,.
\]

2) Let $0<r<1.$ The set $\Omega_{k}\subset\mathbb{R}^{n}$ is then chosen as%
\[
\Omega_{k}=\left\{  x\in\mathbb{R}^{n}:r<\left\vert x\right\vert _{k}<1\text{
and }0<x_{n-k+2},\cdots,x_{n}<1\right\}  .
\]

3) We finally let $\lambda\in C^{1}\left(  \left[  r,1\right]  \right)  $ to
be chosen below,%
\[
\varphi_{k}\left(  x\right)  =\sum_{i=1}^{n-k+1}x_{i}\,dx^{i}\quad
\text{and}\quad\omega_{k}\left(  x\right)  =\lambda\left(  \left\vert
x\right\vert _{k}\right)  \varphi_{k}\left(  x\right)  \wedge dx^{n-k+2}%
\wedge\cdots\wedge dx^{n}\in\Lambda^{k}.
\]

\emph{Step 2.2.}\textbf{ }Observe the following facts.\smallskip

(i) If $\nu$ is the outward unit normal to $\Omega_{k}\,,$ then%
\[
\nu\wedge\omega_{k}=0\quad\text{on }\partial\Omega.
\]
Indeed one sees that this is the case by distinguishing between the lateral
boundaries $\left\vert x\right\vert _{k}=r,1$ where%
\[
\nu=\frac{\pm1}{\left\vert x\right\vert _{k}}\left(  x_{1},\cdots
,x_{n-k+1},0,\cdots,0\right)  \quad\Rightarrow\quad\nu\wedge\varphi_{k}=0
\]
and the horizontal boundaries $x_{s}=0,1$ ($n-k+2\leq s\leq n$) where%
\[
\nu=\pm e_{s}\quad\Rightarrow\quad\nu\wedge dx^{n-k+2}\wedge\cdots\wedge
dx^{n}=0.
\]

(ii) We have, for $1\leq i_{1}<\cdots<i_{k}\leq n,$ that%
\[
\omega_{k}^{i_{1}\cdots i_{k}}\left(  x\right)  =\left\{
\begin{array}
[c]{cl}%
\lambda\left(  \sqrt{x_{1}^{2}+\cdots+x_{n-k+1}^{2}}\,\right)  x_{i_{1}} &
\text{if }%
\begin{array}
[c]{l}%
1\leq i_{1}\leq n-k+1\smallskip\\
\left(  i_{2},\cdots,i_{k}\right)  =\left(  \left(  n-k+2\right)
,\cdots,n\right)
\end{array}
\medskip\\
0 & \text{otherwise.}%
\end{array}
\right.
\]
and thus, if $1\leq i,j\leq n-k+1,$%
\[
\frac{\partial\omega_{k}^{i\,\left(  n-k+2\right)  \cdots n}}{\partial x_{j}%
}=\lambda\,\delta^{ij}+\lambda^{\prime}\,\frac{x_{i}x_{j}}{\left\vert
x\right\vert _{k}}%
\]
and all the other partial derivatives are $0.$\smallskip

(iii) This leads to $d\omega_{k}=0.$ Indeed if we set%
\[
\mu^{\prime}\left(  s\right)  =s\,\lambda\left(  s\right)  \quad
\text{and}\quad\eta\left(  x\right)  =\mu\left(  \left\vert x\right\vert
_{k}\right)
\]
we see that%
\[
\lambda\left(  \left\vert x\right\vert _{k}\right)  \varphi_{k}\left(
x\right)  =d\eta\left(  x\right)  \quad\Rightarrow\quad d\omega_{k}=0.
\]

(iv) We now prove that%
\[
\left\vert \delta\omega_{k}\right\vert ^{2}=\left(  \left(  n-k+1\right)
\lambda+\left\vert x\right\vert _{k}\lambda^{\prime}\right)  ^{2}.
\]
Indeed, since%
\[
\left(  \delta\omega_{k}\right)  ^{i_{1}\cdots i_{k-1}}=\sum_{\gamma=1}%
^{k}\left(  -1\right)  ^{\gamma-1}\sum_{i_{\gamma-1}<j<i_{\gamma}}%
\frac{\partial\omega_{k}^{i_{1}\cdots i_{\gamma-1}ji_{\gamma}\cdots i_{k-1}}%
}{\partial x_{j}}\,,
\]
we have $\left(  \delta\omega_{k}\right)  ^{i_{1}\cdots i_{k-1}}=0$ unless
$\left(  i_{1},\cdots,i_{k-1}\right)  =\left(  \left(  n-k+2\right)
,\cdots,n\right)  ;$ while%
\[
\left(  \delta\omega_{k}\right)  ^{(n-k+2)\cdots n}=\sum_{j=1}^{n-k+1}%
\frac{\partial\omega_{k}^{jn-k+2\cdots n}}{\partial x_{j}}=\sum_{j=1}%
^{n-k+1}\frac{\partial\left(  \lambda\left(  \left\vert x\right\vert
_{k}\right)  x_{j}\right)  }{\partial x_{j}}=\left(  n-k+1\right)
\lambda+\left\vert x\right\vert _{k}\lambda^{\prime}.
\]

(v) We next observe that%
\[
\left\vert \nabla\omega_{k}\right\vert ^{2}=\left(  n-k+1\right)  \lambda
^{2}+2\left\vert x\right\vert _{k}\lambda\,\lambda^{\prime}+\left\vert
x\right\vert _{k}^{2}\left(  \lambda^{\prime}\right)  ^{2}.
\]

(vi) Finally choose $\lambda\left(  s\right)  =s^{-\left(  n-k+1\right)  }$
(with this choice we have $\delta\omega_{k}=0$). We therefore have
($\sigma_{n-k+1}$ denoting the measure of the unit sphere of $\mathbb{R}%
^{n-k+1}$) that%
\begin{align*}%
{\displaystyle\int_{\Omega_{k}}}
\left\vert \nabla\omega_{k}\right\vert ^{2}  &  =\sigma_{n-k+1}\int_{r}%
^{1}\left(  n-k+1\right)  \left(  n-k\right)  s^{-n+k-2}ds=\left.
\sigma_{n-k+1}\left(  k-n\right)  s^{-n+k-1}\right\vert _{r}^{1}\smallskip\\
&  =\sigma_{n-k+1}\left(  n-k\right)  \left[  r^{-n+k-1}-1\right]
\end{align*}
while%
\[%
{\displaystyle\int_{\Omega_{k}}}
\left(  \left\vert d\omega_{k}\right\vert ^{2}+\left\vert \delta\omega
_{k}\right\vert ^{2}+\left\vert \omega_{k}\right\vert ^{2}\right)
=\sigma_{n-k+1}\int_{r}^{1}s^{-n+k}ds=\left\{
\begin{array}
[c]{cl}%
\frac{\sigma_{n-k+1}}{n-k-1}\left[  r^{-n+k+1}-1\right]  & \text{if
}n>k+1\smallskip\\
-\sigma_{n-k+1}\log r & \text{if }n=k+1.
\end{array}
\right.
\]
Therefore, when $r\rightarrow0,$ we find (writing $\sim$ for the asymptotic
behavior)%
\[
\frac{\left\Vert \nabla\omega_{k}\right\Vert ^{2}}{\left\Vert d\omega
_{k}\right\Vert ^{2}+\left\Vert \delta\omega_{k}\right\Vert ^{2}+\left\Vert
\omega_{k}\right\Vert ^{2}}\sim\left\{
\begin{array}
[c]{cl}%
\frac{\left(  n-k-1\right)  \left(  n-k\right)  }{r^{2}} & \text{if
}n>k+1\smallskip\\
-\frac{1}{r^{2}\log r} & \text{if }n=k+1.
\end{array}
\right.
\]
Thus, for $r$ sufficiently small, we deduce that $C_{T}\left(  \Omega
_{k},k\right)  \ $is arbitrarily large as wished.\smallskip
\end{proof}

\section{The case of polytopes}

\begin{definition}
\label{Definition: polytope generalise}$\Omega\subset\mathbb{R}^{n}$ is said
to be a \emph{generalized polytope}, if there exist $\Omega_{0}\,,\Omega
_{1}\,,\cdots,\Omega_{M}$ bounded open polytopes such that, for every
$i,j=1,\cdots,M$ with $i\neq j,$%
\[
\overline{\Omega}_{i}\subset\Omega_{0}\,,\quad\overline{\Omega}_{i}%
\cap\overline{\Omega}_{j}=\emptyset\quad\text{and}\quad\Omega=\Omega
_{0}\setminus\left(
{\textstyle\bigcup\limits_{i=1}^{M}}
\overline{\Omega}_{i}\right)  .
\]
In this case $\overline{\Omega}_{i}\,,i=1,\cdots,M,$ are called the
\emph{holes}.
\end{definition}

\begin{theorem}
\label{Theoreme polytopes}Let $\Omega\subset\mathbb{R}^{n}$ be a generalized
polytope. Then the following identity holds%
\[
\left\Vert \nabla\omega\right\Vert ^{2}=\left\Vert d\omega\right\Vert
^{2}+\left\Vert \delta\omega\right\Vert ^{2},\quad\forall\,\omega\in C_{T}%
^{1}\left(  \overline{\Omega};\Lambda^{k}\right)  \cup C_{N}^{1}\left(
\overline{\Omega};\Lambda^{k}\right)  .
\]

\end{theorem}

\begin{remark}
\label{Remarque: champs harmoniques et polytopes}\emph{(i)} Note that we do
not make any assumption on the topology of the domain and that holes are
allowed. The identity shows that there are no non-trivial harmonic fields with
vanishing tangential (or normal) component which are of class $C^{1}.$
However, in presence of holes, there are non-trivial harmonic fields with
weaker regularity (this is, of course, a problem only on the boundary, since
harmonic fields are $C^{\infty}$ in the interior).

\emph{(ii)} In the case $k=1,$ see also \cite{Csato-Kneuss-Rajendran}.
\end{remark}

Before proceeding with the proof, we need to introduce a few notations that
would help us keep track of the signs in the proof.

\begin{notation}
\emph{(i)} For $1\leqslant k\leqslant n,$ we write%
\[
\mathcal{T}^{k}=\{(i_{1},\cdots,i_{k})\in\mathbb{N}^{k}:1\leqslant
i_{1}<\cdots<i_{k}\leqslant n\}.
\]
For $I=(i_{1},\cdots,i_{k})\in\mathcal{T}^{k},$ we write $dx^{I}$ to denote
$dx^{i_{1}}\wedge\cdots\wedge dx^{i_{k}}.$\smallskip

\emph{(ii)} For $i\in I,$ we write $I_{\widehat{i}}=(i_{1},\cdots
,\widehat{i},\cdots,i_{k}),$ where $\widehat{i}$ denotes the absence of the
named index $i.$ Note that, $I_{\widehat{i_{p}}}\in\mathcal{T}^{k-1}\,,$ for
all $1\leqslant p\leqslant k.$ Similarly, for $i,j\in I,$ $i<j,$ we write
$I_{\widehat{ij}}=(i_{1},\cdots,\widehat{i},\cdots,\widehat{j},\cdots,i_{k}%
).$\smallskip

\emph{(iii)} Given $I\in\mathcal{T}^{k}$ and $i,j\notin I,$ $i\neq j,$ we
write $\left[  iI\right]  $ to denote the increasing multiindex formed by the
index $i$ and the indices in $I.$ In other words $\left[  iI\right]  $ is the
permutation of the indices such that $\left[  iI\right]  \in\mathcal{T}%
^{k+1}.$ Furthermore, we define the sign of $\left[  i,I\right]  ,$ denoted by
$\operatorname*{sgn}\left[  i,I\right]  ,$ as%
\[
dx^{\left[  iI\right]  }=\operatorname*{sgn}\left[  i,I\right]  dx^{i}\wedge
dx^{I}.
\]
Similarly, $\left[  ijI\right]  $ is the permutation of the indices such that
$[ijI]\in\mathcal{T}^{k+2}$ and $\operatorname*{sgn}\left[  i,j,I\right]  $ is
given by%
\[
dx^{\left[  ijI\right]  }=\operatorname*{sgn}\left[  i,j,I\right]
dx^{i}\wedge dx^{j}\wedge dx^{I}.
\]

\end{notation}

We need a few lemmas for the theorem.

\begin{lemma}
\label{sign lemma} Let $n\geq2$ and $1\leq k\leq n-1$ be integers. Then for
any $I\in\mathcal{T}^{k+1}$ and every $i,j\in I$ with $i\neq j,$%
\[
\operatorname*{sgn}\left[  i,I_{\widehat{ij}}\right]  \operatorname*{sgn}%
\left[  j,I_{\widehat{ij}}\right]  =-\operatorname*{sgn}\left[
i,I_{\widehat{i}}\right]  \operatorname*{sgn}\left[  j,I_{\widehat{j}}\right]
\]
and for any $I\in\mathcal{T}^{k-1}$ and every $i,j\notin I$ with $i\neq j,$%
\[
\operatorname*{sgn}\left[  i,\left[  jI\right]  \right]  \operatorname*{sgn}%
\left[  j,\left[  iI\right]  \right]  =-\operatorname*{sgn}\left[  i,I\right]
\operatorname*{sgn}\left[  j,I\right]  .
\]

\end{lemma}

\begin{remark}
When $k=1$ both equations read as%
\[
\operatorname*{sgn}\left[  i,j\right]  \operatorname*{sgn}\left[  j,i\right]
=-1.
\]
Indeed elements $I\in\mathcal{T}^{k-1}$ are as if they were absent, i.e.
$\left[  jI\right]  =j$ and $\operatorname*{sgn}\left[  i,I\right]  =1.$
\end{remark}

\begin{proof}
Since $I_{\widehat{i}}=\left[  jI_{\widehat{ij}}\right]  $ and $I_{\widehat{j}%
}=\left[  iI_{\widehat{ij}}\right]  ,$ we have the identity%
\[
\operatorname*{sgn}\left[  i,I_{\widehat{i}}\right]  \operatorname*{sgn}%
\left[  j,I_{\widehat{ij}}\right]  =\operatorname*{sgn}\left[
i,j,I_{\widehat{ij}}\right]  =-\operatorname*{sgn}\left[  j,i,I_{\widehat{ij}%
}\right]  =-\operatorname*{sgn}\left[  j,I_{\widehat{j}}\right]
\operatorname*{sgn}\left[  i,I_{\widehat{ij}}\right]  .
\]
This proves the first identity. The second one is just the first one, where
$I\in\mathcal{T}^{k-1}$ plays the role of $I_{\widehat{ij}}\,.$ This finishes
the proof.\smallskip
\end{proof}

The next lemma gives a pointwise identity.

\begin{lemma}
\label{algebraic identity} Let $U\subset\mathbb{R}^{n}$ be open and let
$\omega\in C^{1}\left(  U;\Lambda^{k}\right)  .$ Then, for any $x\in U,$%
\begin{equation}
\left\vert d\omega\right\vert ^{2}+\left\vert \delta\omega\right\vert
^{2}-\left\vert \nabla\omega\right\vert ^{2}=\sum_{I\in\mathcal{T}^{k+1}}%
\sum_{\substack{i,j\in I\\i\neq j}}\operatorname*{sgn}\left[  i,I_{\widehat{i}%
}\right]  \operatorname*{sgn}\left[  j,I_{\widehat{j}}\right]  \left(
\frac{\partial\omega^{I_{\widehat{i}}}}{\partial x_{i}}\frac{\partial
\omega^{I_{\widehat{j}}}}{\partial x_{j}}-\frac{\partial\omega^{I_{\widehat{j}%
}}}{\partial x_{i}}\frac{\partial\omega^{I_{\widehat{i}}}}{\partial x_{j}%
}\right)  \label{pointwise identiy k+1}%
\end{equation}%
\begin{equation}
\left\vert d\omega\right\vert ^{2}+\left\vert \delta\omega\right\vert
^{2}-\left\vert \nabla\omega\right\vert ^{2}=\sum_{I\in\mathcal{T}^{k-1}}%
\sum_{\substack{i,j\notin I\\i<j}}\operatorname*{sgn}\left[  i,I\right]
\operatorname*{sgn}\left[  j,I\right]  \left(  \frac{\partial\omega^{\left[
iI\right]  }}{\partial x_{i}}\frac{\partial\omega^{\left[  jI\right]  }%
}{\partial x_{j}}-\frac{\partial\omega^{\left[  jI\right]  }}{\partial x_{i}%
}\frac{\partial\omega^{\left[  iI\right]  }}{\partial x_{j}}\right)  .
\label{pointwise identiy k-1}%
\end{equation}

\end{lemma}

\begin{remark}
When $k=1$ the lemma reads as%
\[
\left\vert d\omega\right\vert ^{2}+\left\vert \delta\omega\right\vert
^{2}-\left\vert \nabla\omega\right\vert ^{2}=2\sum_{i<j}\left(  \frac
{\partial\omega^{i}}{\partial x_{i}}\frac{\partial\omega^{j}}{\partial x_{j}%
}-\frac{\partial\omega^{j}}{\partial x_{i}}\frac{\partial\omega^{i}}{\partial
x_{j}}\right)
\]
while for $k=2$%
\begin{align*}
&  \left\vert d\omega\right\vert ^{2}+\left\vert \delta\omega\right\vert
^{2}-\left\vert \nabla\omega\right\vert ^{2}\smallskip\\
&  =2\sum_{i<j<k}\left(  \frac{\partial\omega^{ij}}{\partial x_{j}}%
\frac{\partial\omega^{ik}}{\partial x_{k}}-\frac{\partial\omega^{ij}}{\partial
x_{k}}\frac{\partial\omega^{ik}}{\partial x_{j}}+\frac{\partial\omega^{ij}%
}{\partial x_{k}}\frac{\partial\omega^{jk}}{\partial x_{i}}-\frac
{\partial\omega^{ij}}{\partial x_{i}}\frac{\partial\omega^{jk}}{\partial
x_{k}}+\frac{\partial\omega^{ik}}{\partial x_{i}}\frac{\partial\omega^{jk}%
}{\partial x_{j}}-\frac{\partial\omega^{ik}}{\partial x_{j}}\frac
{\partial\omega^{jk}}{\partial x_{i}}\right)  .
\end{align*}

\end{remark}

\begin{proof}
We calculate%
\[
d\omega=\sum_{I\in\mathcal{T}^{k+1}}\left(  \sum_{i\in I}\operatorname*{sgn}%
\left[  i,I_{\widehat{i}}\right]  \frac{\partial\omega^{I_{\widehat{i}}}%
}{\partial x_{i}}\right)  dx^{I}\quad\text{and}\quad\delta\omega=\sum
_{I\in\mathcal{T}^{k-1}}\left(  \sum_{i\notin I}\operatorname*{sgn}\left[
i,I\right]  \frac{\partial\omega^{\left[  iI\right]  }}{\partial x_{i}%
}\right)  dx^{I}.
\]
We start by evaluating $\left\vert d\omega\right\vert ^{2},$ we find%
\[
\left\vert d\omega\right\vert ^{2}=\sum_{I\in\mathcal{T}^{k+1}}\sum_{i\in
I}\left(  \frac{\partial\omega^{I_{\widehat{i}}}}{\partial x_{i}}\right)
^{2}+2\sum_{I\in\mathcal{T}^{k+1}}\sum_{\substack{i,j\in I\\i<j}%
}\operatorname*{sgn}\left[  i,I_{\widehat{i}}\right]  \operatorname*{sgn}%
\left[  j,I_{\widehat{j}}\right]  \frac{\partial\omega^{I_{\widehat{i}}}%
}{\partial x_{i}}\frac{\partial\omega^{I_{\widehat{j}}}}{\partial x_{j}}\,.
\]
Rewriting the terms in two different ways, we obtain,%
\begin{equation}
\left\vert d\omega\right\vert ^{2}=\sum_{I\in\mathcal{T}^{k}}\sum_{i\notin
I}\left(  \frac{\partial\omega^{I}}{\partial x_{i}}\right)  ^{2}+2\sum
_{I\in\mathcal{T}^{k+1}}\sum_{\substack{i,j\in I\\i<j}}\operatorname*{sgn}%
\left[  i,I_{\widehat{i}}\right]  \operatorname*{sgn}\left[  j,I_{\widehat{j}%
}\right]  \frac{\partial\omega^{I_{\widehat{i}}}}{\partial x_{i}}%
\frac{\partial\omega^{I_{\widehat{j}}}}{\partial x_{j}} \label{k+1d}%
\end{equation}%
\begin{equation}
\left\vert d\omega\right\vert ^{2}=\sum_{I\in\mathcal{T}^{k}}\sum_{i\notin
I}\left(  \frac{\partial\omega^{I}}{\partial x_{i}}\right)  ^{2}+2\sum
_{I\in\mathcal{T}^{k-1}}\sum_{\substack{i,j\notin I\\i<j}}\operatorname*{sgn}%
\left[  i,\left[  jI\right]  \right]  \operatorname*{sgn}\left[  j,\left[
iI\right]  \right]  \frac{\partial\omega^{\left[  jI\right]  }}{\partial
x_{i}}\frac{\partial\omega^{\left[  iI\right]  }}{\partial x_{j}}.
\label{k-1d}%
\end{equation}
We next evaluate $\left\vert \delta\omega\right\vert ^{2},$ we get%
\[
\left\vert \delta\omega\right\vert ^{2}=\sum_{I\in\mathcal{T}^{k-1}}%
\sum_{i\notin I}\left(  \frac{\partial\omega^{\left[  iI\right]  }}{\partial
x_{i}}\right)  ^{2}+2\sum_{I\in\mathcal{T}^{k-1}}\sum_{\substack{i,j\notin
I\\i<j}}\operatorname*{sgn}\left[  i,I\right]  \operatorname*{sgn}\left[
j,I\right]  \frac{\partial\omega^{\left[  iI\right]  }}{\partial x_{i}}%
\frac{\partial\omega^{\left[  jI\right]  }}{\partial x_{j}}.
\]
Rewriting the terms in two different ways, we find%
\begin{equation}
\left\vert \delta\omega\right\vert ^{2}=\sum_{I\in\mathcal{T}^{k}}\sum_{i\in
I}\left(  \frac{\partial\omega^{I}}{\partial x_{i}}\right)  ^{2}+2\sum
_{I\in\mathcal{T}^{k-1}}\sum_{\substack{i,j\notin I\\i<j}}\operatorname*{sgn}%
\left[  i,I\right]  \operatorname*{sgn}\left[  j,I\right]  \frac
{\partial\omega^{\left[  iI\right]  }}{\partial x_{i}}\frac{\partial
\omega^{\left[  jI\right]  }}{\partial x_{j}} \label{k-1delta}%
\end{equation}%
\begin{equation}
\left\vert \delta\omega\right\vert ^{2}=\sum_{I\in\mathcal{T}^{k}}\sum_{i\in
I}\left(  \frac{\partial\omega^{I}}{\partial x_{i}}\right)  ^{2}+2\sum
_{I\in\mathcal{T}^{k+1}}\sum_{\substack{i,j\in I\\i<j}}\operatorname*{sgn}%
\left[  i,I_{\widehat{ij}}\right]  \operatorname*{sgn}\left[
j,I_{\widehat{ij}}\right]  \frac{\partial\omega^{I_{\widehat{j}}}}{\partial
x_{i}}\frac{\partial\omega^{I_{\widehat{i}}}}{\partial x_{j}}.
\label{k+1delta}%
\end{equation}
Appealing to (\ref{k+1d}) and (\ref{k+1delta}), we infer that%
\begin{align*}
\left\vert d\omega\right\vert ^{2}  &  +\left\vert \delta\omega\right\vert
^{2}-\left\vert \nabla\omega\right\vert ^{2}\\
&  =2\sum_{I\in\mathcal{T}^{k+1}}\sum_{\substack{i,j\in I\\i<j}}\left(
\operatorname*{sgn}\left[  i,I_{\widehat{i}}\right]  \operatorname*{sgn}%
\left[  j,I_{\widehat{j}}\right]  \frac{\partial\omega^{I_{\widehat{i}}}%
}{\partial x_{i}}\frac{\partial\omega^{I_{\widehat{j}}}}{\partial x_{j}%
}+\operatorname*{sgn}\left[  i,I_{\widehat{ij}}\right]  \operatorname*{sgn}%
\left[  j,I_{\widehat{ij}}\right]  \frac{\partial\omega^{I_{\widehat{j}}}%
}{\partial x_{i}}\frac{\partial\omega^{I_{\widehat{i}}}}{\partial x_{j}%
}\right)  .
\end{align*}
Invoking (\ref{k-1d}) and (\ref{k-1delta}), we find%
\begin{align*}
\left\vert d\omega\right\vert ^{2}  &  +\left\vert \delta\omega\right\vert
^{2}-\left\vert \nabla\omega\right\vert ^{2}\smallskip\\
&  =\sum_{I\in\mathcal{T}^{k-1}}\sum_{\substack{i,j\notin I\\i<j}}\left(
\operatorname*{sgn}\left[  i,I\right]  \operatorname*{sgn}\left[  j,I\right]
\frac{\partial\omega^{\left[  iI\right]  }}{\partial x_{i}}\frac
{\partial\omega^{\left[  jI\right]  }}{\partial x_{j}}+\operatorname*{sgn}%
\left[  i,\left[  jI\right]  \right]  \operatorname*{sgn}\left[  j,\left[
iI\right]  \right]  \frac{\partial\omega^{\left[  jI\right]  }}{\partial
x_{i}}\frac{\partial\omega^{\left[  iI\right]  }}{\partial x_{j}}\right)  .
\end{align*}
Using Lemma \ref{sign lemma}, the last two identities establish
(\ref{pointwise identiy k+1}) and (\ref{pointwise identiy k-1}),
respectively.\smallskip
\end{proof}

\begin{lemma}
\label{afterintegrationbypartslipschitz} Let $\Omega\subset\mathbb{R}^{n}$ be
a bounded open Lipschitz set and $\omega\in C^{1}\left(  \overline{\Omega
};\Lambda^{k}\right)  .$ Then%
\begin{equation}
\left\Vert d\omega\right\Vert ^{2}+\left\Vert \delta\omega\right\Vert
^{2}-\left\Vert \nabla\omega\right\Vert ^{2}=\int_{\partial\Omega}\left\langle
\nu\wedge\omega;d\omega\right\rangle -\int_{\partial\Omega}\sum_{I\in
\mathcal{T}^{k+1}}\sum_{i,j\in I}\operatorname*{sgn}\left[  i,I_{\widehat{i}%
}\right]  \operatorname*{sgn}\left[  j,I_{\widehat{j}}\right]  \omega
^{I_{\widehat{j}}}\,\nu^{i}\,\frac{\partial\omega^{I_{\widehat{i}}}}{\partial
x_{j}} \label{tangent k+1 version}%
\end{equation}%
\begin{equation}
\left\Vert d\omega\right\Vert ^{2}+\left\Vert \delta\omega\right\Vert
^{2}-\left\Vert \nabla\omega\right\Vert ^{2}=\int_{\partial\Omega}\left\langle
\nu\,\lrcorner\,\omega;\delta\omega\right\rangle -\int_{\partial\Omega}%
\sum_{I\in\mathcal{T}^{k-1}}\sum_{i,j\notin I}\operatorname*{sgn}\left[
i,I\right]  \operatorname*{sgn}\left[  j,I\right]  \omega^{\left[  iI\right]
}\,\nu^{j}\,\frac{\partial\omega^{\left[  jI\right]  }}{\partial x_{i}}\,.
\label{normal k-1 version}%
\end{equation}

\end{lemma}

\begin{proof}
We divide the proof in three steps.\smallskip

\emph{Step 1.} Integrate the equations of Lemma \ref{algebraic identity} to
get%
\begin{equation}
\left\Vert d\omega\right\Vert ^{2}+\left\Vert \delta\omega\right\Vert
^{2}-\left\Vert \nabla\omega\right\Vert ^{2}=\int_{\Omega}\sum_{I\in
\mathcal{T}^{k+1}}\sum_{\substack{i,j\in I\\i\neq j}}\operatorname*{sgn}%
\left[  i,I_{\widehat{i}}\right]  \operatorname*{sgn}\left[  j,I_{\widehat{j}%
}\right]  \left(  \frac{\partial\omega^{I_{\widehat{i}}}}{\partial x_{i}}%
\frac{\partial\omega^{I_{\widehat{j}}}}{\partial x_{j}}-\frac{\partial
\omega^{I_{\widehat{j}}}}{\partial x_{i}}\frac{\partial\omega^{I_{\widehat{i}%
}}}{\partial x_{j}}\right)  \label{integrated version k+1}%
\end{equation}%
\begin{equation}
\left\Vert d\omega\right\Vert ^{2}+\left\Vert \delta\omega\right\Vert
^{2}-\left\Vert \nabla\omega\right\Vert ^{2}=\int_{\Omega}\sum_{I\in
\mathcal{T}^{k-1}}\sum_{\substack{i,j\notin I\\i<j}}\operatorname*{sgn}\left[
i,I\right]  \operatorname*{sgn}\left[  j,I\right]  \left(  \frac
{\partial\omega^{\left[  iI\right]  }}{\partial x_{i}}\frac{\partial
\omega^{\left[  jI\right]  }}{\partial x_{j}}-\frac{\partial\omega^{\left[
jI\right]  }}{\partial x_{i}}\frac{\partial\omega^{\left[  iI\right]  }%
}{\partial x_{j}}\right)  . \label{integrated version k-1}%
\end{equation}

\emph{Step 2.} Noting that (the following argument uses the fact that
$\omega\in C^{2}$ but by density the identity
(\ref{afterintegrationbypartslemma}) is valid for $\omega\in C^{1}$)%
\[
\frac{\partial\omega^{I_{\widehat{i}}}}{\partial x_{i}}\frac{\partial
\omega^{I_{\widehat{j}}}}{\partial x_{j}}-\frac{\partial\omega^{I_{\widehat{j}%
}}}{\partial x_{i}}\frac{\partial\omega^{I_{\widehat{i}}}}{\partial x_{j}%
}=\frac{\partial}{\partial x_{j}}\left(  \omega^{I_{\widehat{j}}}%
\,\frac{\partial\omega^{I_{\widehat{i}}}}{\partial x_{i}}\right)
-\frac{\partial}{\partial x_{i}}\left(  \omega^{I_{\widehat{j}}}%
\,\frac{\partial\omega^{I_{\widehat{i}}}}{\partial x_{j}}\right)
\]
we integrate by parts (\ref{integrated version k+1}), bearing in mind that
$\Omega$ is Lipschitz, to obtain%
\begin{equation}
\left\Vert d\omega\right\Vert ^{2}+\left\Vert \delta\omega\right\Vert
^{2}-\left\Vert \nabla\omega\right\Vert ^{2}=\int_{\partial\Omega}\sum
_{I\in\mathcal{T}^{k+1}}\sum_{\substack{i,j\in I\\i\neq j}}\operatorname*{sgn}%
\left[  i,I_{\widehat{i}}\right]  \operatorname*{sgn}\left[  j,I_{\widehat{j}%
}\right]  \left(  \omega^{I_{\widehat{j}}}\,\nu^{j}\,\frac{\partial
\omega^{I_{\widehat{i}}}}{\partial x_{i}}-\omega^{I_{\widehat{j}}}\,\nu
^{i}\,\frac{\partial\omega^{I_{\widehat{i}}}}{\partial x_{j}}\right)  .
\label{afterintegrationbypartslemma}%
\end{equation}

\emph{Step 3.} Since $\Omega$ is Lipschitz, $\nu$ is defined for a.e.
$x\in\partial\Omega.$ Thus, for a.e. $x\in\partial\Omega,$ we find%
\[
\nu\wedge\omega=\sum_{I\in\mathcal{T}^{k+1}}\left(  \sum_{j\in I}%
\operatorname*{sgn}\left[  j,I_{\widehat{j}}\right]  \nu^{j}\,\omega
^{I_{\widehat{j}}}\right)  dx^{I}.
\]
But, since $\omega\in C^{1}\left(  \overline{\Omega};\Lambda^{k}\right)  ,$
for every $x\in\overline{\Omega},$ we have%
\[
d\omega=\sum_{I\in\mathcal{T}^{k+1}}\left(  \sum_{i\in I}\operatorname*{sgn}%
\left[  i,I_{\widehat{i}}\right]  \frac{\partial\omega^{I_{\widehat{i}}}%
}{\partial x_{i}}\right)  dx^{I}.
\]
We therefore deduce the pointwise identity, for a.e. $x\in\partial\Omega,$%
\begin{align*}
\left\langle \nu\wedge\omega;d\omega\right\rangle  &  =\sum_{I\in
\mathcal{T}^{k+1}}\sum_{i,j\in I}\operatorname*{sgn}\left[  i,I_{\widehat{i}%
}\right]  \operatorname*{sgn}\left[  j,I_{\widehat{j}}\right]  \omega
^{I_{\widehat{j}}}\,\nu^{j}\,\frac{\partial\omega^{I_{\widehat{i}}}}{\partial
x_{i}}\smallskip\\
&  =\sum_{I\in\mathcal{T}^{k+1}}\sum_{i\in I}\omega^{I_{\widehat{i}}}\,\nu
^{i}\,\frac{\partial\omega^{I_{\widehat{i}}}}{\partial x_{i}}+\sum
_{I\in\mathcal{T}^{k+1}}\sum_{\substack{i,j\in I\\i\neq j}}\operatorname*{sgn}%
\left[  i,I_{\widehat{i}}\right]  \operatorname*{sgn}\left[  j,I_{\widehat{j}%
}\right]  \omega^{I_{\widehat{j}}}\,\nu^{j}\,\frac{\partial\omega
^{I_{\widehat{i}}}}{\partial x_{i}}\,.
\end{align*}
We then have%
\begin{equation}
\sum_{I\in\mathcal{T}^{k+1}}\sum_{\substack{i,j\in I\\i\neq j}%
}\operatorname*{sgn}\left[  i,I_{\widehat{i}}\right]  \operatorname*{sgn}%
\left[  j,I_{\widehat{j}}\right]  \omega^{I_{\widehat{j}}}\,\nu^{j}%
\,\frac{\partial\omega^{I_{\widehat{i}}}}{\partial x_{i}}=\left\langle
\nu\wedge\omega;d\omega\right\rangle -\sum_{I\in\mathcal{T}^{k+1}}\sum_{i\in
I}\omega^{I_{\widehat{i}}}\,\nu^{i}\,\frac{\partial\omega^{I_{\widehat{i}}}%
}{\partial x_{i}}\,. \label{first term byparts}%
\end{equation}
Substituting (\ref{first term byparts}) in (\ref{afterintegrationbypartslemma}%
), we obtain (\ref{tangent k+1 version}). Analogous calculations yield
(\ref{normal k-1 version}), starting from integration by parts of
(\ref{integrated version k-1}).\smallskip
\end{proof}

We now prove a theorem for piecewise $C^{2}$ Lipschitz domains, which can be
viewed as a generalization of Theorem 3.1.1.2 in \cite{Grisvard 1985} valid
for $k=1.$

\begin{theorem}
\label{piecewisesmoothformula}Let $\Omega\subset\mathbb{R}^{n}$ be a bounded
open Lipschitz set with piecewise $C^{2}$ boundary, i.e. $\partial
\Omega\smallskip=\cup_{s=1}^{N}\overline{\Gamma_{s}}\,,$ where the $\Gamma
_{s}$ are $C^{2}$ and relatively open subset of $\partial\Omega$ and
$\partial\Omega\setminus\cup_{s=1}^{N}\Gamma_{s}$\smallskip\ has zero surface
measure. Let $E_{1},\cdots,E_{n-1}$ be an orthonormal frame field of principal
directions of $\cup_{s=1}^{N}\Gamma_{s}$\smallskip\ with associated principal
curvatures $\gamma_{1},\cdots,\gamma_{n-1}\,.$ Then%
\begin{equation}
\left\Vert d\omega\right\Vert ^{2}+\left\Vert \delta\omega\right\Vert
^{2}-\left\Vert \nabla\omega\right\Vert ^{2}=\sum_{s=1}^{N}\int_{\Gamma_{s}%
}\left(  \sum_{l=1}^{n-1}\gamma_{l}\left\vert E_{l}\wedge\omega\right\vert
^{2}\right)  \quad\text{for every }\omega\in C_{T}^{1}\left(  \overline
{\Omega};\Lambda^{k}\right)  \label{tangentialcondition}%
\end{equation}%
\begin{equation}
\left\Vert d\omega\right\Vert ^{2}+\left\Vert \delta\omega\right\Vert
^{2}-\left\Vert \nabla\omega\right\Vert ^{2}=\sum_{s=1}^{N}\int_{\Gamma_{s}%
}\left(  \sum_{l=1}^{n-1}\gamma_{l}\left\vert E_{l}\,\lrcorner\,\omega
\right\vert ^{2}\right)  \quad\text{for every }\omega\in C_{N}^{1}\left(
\overline{\Omega};\Lambda^{k}\right)  . \label{normalcondition}%
\end{equation}

\end{theorem}

\begin{remark}
\emph{(i)} By standard regularization, the theorem is valid for $\omega\in
W_{T}^{1,2}\left(  \Omega;\Lambda^{k}\right)  $ (respectively $\omega\in
W_{N}^{1,2}\left(  \Omega;\Lambda^{k}\right)  $) if $\partial\Omega$ is
(fully) $C^{2}.$\smallskip

\emph{(ii)} Note that the two identities above are the same as those appearing
at the end of Theorem \ref{thm:fund identity with curvatures}.
\end{remark}

\begin{proof}
\emph{Step 1.} We first show (\ref{tangentialcondition}). Since $\partial
\Omega\setminus\cup_{s=1}^{N}\Gamma_{s}$ has zero surface measure, we obtain
from (\ref{tangent k+1 version})%
\begin{equation}
\left.
\begin{array}
[c]{l}%
\left\Vert d\omega\right\Vert ^{2}+\left\Vert \delta\omega\right\Vert
^{2}-\left\Vert \nabla\omega\right\Vert ^{2}\smallskip\\
=%
{\displaystyle\int_{\partial\Omega}}
\left\langle \nu\wedge\omega;d\omega\right\rangle -%
{\displaystyle\sum_{s=1}^{N}}
{\displaystyle\int_{\Gamma_{s}}}
{\displaystyle\sum_{I\in\mathcal{T}^{k+1}}}
{\displaystyle\sum_{i,j\in I}}
\operatorname*{sgn}\left[  i,I_{\widehat{i}}\right]  \operatorname*{sgn}%
\left[  j,I_{\widehat{j}}\right]  \omega^{I_{\widehat{j}}}\,\nu^{i}%
\,\frac{\partial\omega^{I_{\widehat{i}}}}{\partial x_{j}}\,.
\end{array}
\right.  \label{tangent equation}%
\end{equation}
We argue on each of the $\Gamma_{s}\,.$ Since $E_{1},\cdots,E_{n-1}$ denote a
frame of tangent vectors at each point of $\Gamma_{s}\,,$ we can write
$e_{j}=\sum_{l=1}^{n-1}E_{l}^{j}\,E_{l}+\nu^{j}\,\nu.$ Thus, for any
$j=1,\cdots,n,$%
\begin{equation}
\frac{\partial\omega^{I_{\widehat{i}}}}{\partial x_{j}}=\sum_{l=1}^{n-1}%
E_{l}^{j}\,\frac{\partial\omega^{I_{\widehat{i}}}}{\partial E_{l}}+\nu
^{j}\,\frac{\partial\omega^{I_{\widehat{i}}}}{\partial\nu}\,.
\label{Equation: derivee en fonction tangente et normale}%
\end{equation}

\emph{Step 1.1.} We set%
\[
A=\sum_{I\in\mathcal{T}^{k+1}}\sum_{i,j\in I}\operatorname*{sgn}\left[
i,I_{\widehat{i}}\right]  \operatorname*{sgn}\left[  j,I_{\widehat{j}}\right]
\omega^{I_{\widehat{j}}}\,\nu^{i}\,\frac{\partial\omega^{I_{\widehat{i}}}%
}{\partial x_{j}}%
\]
and note, in view of (\ref{Equation: derivee en fonction tangente et normale}%
), that $A=B+C$ where%
\[
B=\sum_{I\in\mathcal{T}^{k+1}}\left(  \sum_{i\in I}\operatorname*{sgn}\left[
i,I_{\widehat{i}}\right]  \nu^{i}\,\frac{\partial\omega^{I_{\widehat{i}}}%
}{\partial\nu}\right)  \left(  \sum_{j\in I}\operatorname*{sgn}\left[
j,I_{\widehat{j}}\right]  \nu^{j}\,\omega^{I_{\widehat{j}}}\right)
\]%
\[
C=\sum_{I\in\mathcal{T}^{k+1}}\sum_{i,j\in I}\operatorname*{sgn}\left[
i,I_{\widehat{i}}\right]  \operatorname*{sgn}\left[  j,I_{\widehat{j}}\right]
\omega^{I_{\widehat{j}}}\,\nu^{i}\left(  \sum_{l=1}^{n-1}E_{l}^{j}%
\,\frac{\partial\omega^{I_{\widehat{i}}}}{\partial E_{l}}\right)  .
\]

(i) We first observe that%
\[
B=\sum_{I\in\mathcal{T}^{k+1}}\left(  \nu\wedge\frac{\partial\omega}%
{\partial\nu}\right)  ^{I}\left(  \nu\wedge\omega\right)  ^{I}=\left\langle
\nu\wedge\omega;\nu\wedge\frac{\partial\omega}{\partial\nu}\right\rangle .
\]

(ii) We next prove that%
\[
C=\sum_{l=1}^{n-1}\left\langle E_{l}\wedge\omega;\frac{\partial}{\partial
E_{l}}\left(  \nu\wedge\omega\right)  \right\rangle -\sum_{l=1}^{n-1}%
\gamma_{l}\left\vert E_{l}\wedge\omega\right\vert ^{2}.
\]
Indeed note that, for any $I\in\mathcal{T}^{k+1}$ and any $l=1,\cdots,n-1,$%
\[
\frac{\partial}{\partial E_{l}}\left(  \nu\wedge\omega\right)  ^{I}%
=\frac{\partial}{\partial E_{l}}\left(  \sum_{i\in I}\operatorname*{sgn}%
\left[  i,I_{\widehat{i}}\right]  \nu^{i}\,\omega^{I_{\widehat{i}}}\right)
=\sum_{i\in I}\operatorname*{sgn}\left[  i,I_{\widehat{i}}\right]  \nu
^{i}\,\frac{\partial\omega^{I_{\widehat{i}}}}{\partial E_{l}}+\sum_{i\in
I}\operatorname*{sgn}\left[  i,I_{\widehat{i}}\right]  \omega^{I_{\widehat{i}%
}}\,\frac{\partial\nu^{i}}{\partial E_{l}}\,.
\]
For any $j\in I,$ multiplying by $\operatorname*{sgn}\left[  j,I_{\widehat{j}%
}\right]  \omega^{I_{\widehat{j}}}E_{l}^{j}$ and summing over $l=1,\cdots,n-1$
and $j\in I,$ we deduce (recalling that $\gamma_{l}\,E_{l}^{i}=\partial\nu
^{i}/\partial E_{l}$) that%
\begin{align*}
C^{I}  &  =\sum_{i,j\in I}\operatorname*{sgn}\left[  i,I_{\widehat{i}}\right]
\operatorname*{sgn}\left[  j,I_{\widehat{j}}\right]  \omega^{I_{\widehat{j}}%
}\,\nu^{i}\left(  \sum_{l=1}^{n-1}E_{l}^{j}\,\frac{\partial\omega
^{I_{\widehat{i}}}}{\partial E_{l}}\right)  \smallskip\\
&  =\sum_{j\in I}\sum_{l=1}^{n-1}\operatorname*{sgn}\left[  j,I_{\widehat{j}%
}\right]  \omega^{I_{\widehat{j}}}E_{l}^{j}\,\frac{\partial}{\partial E_{l}%
}\left(  \nu\wedge\omega\right)  ^{I}-\sum_{i,j\in I}\operatorname*{sgn}%
\left[  i,I_{\widehat{i}}\right]  \operatorname*{sgn}\left[  j,I_{\widehat{j}%
}\right]  \omega^{I_{\widehat{i}}}\omega^{I_{\widehat{j}}}\left(  \sum
_{l=1}^{n-1}E_{l}^{j}\,\frac{\partial\nu^{i}}{\partial E_{l}}\right)
\smallskip\\
&  =\sum_{l=1}^{n-1}\left(  E_{l}\wedge\omega\right)  ^{I}\frac{\partial
}{\partial E_{l}}\left(  \nu\wedge\omega\right)  ^{I}-\sum_{i,j\in
I}\operatorname*{sgn}\left[  i,I_{\widehat{i}}\right]  \operatorname*{sgn}%
\left[  j,I_{\widehat{j}}\right]  \omega^{I_{\widehat{i}}}\omega
^{I_{\widehat{j}}}\left(  \sum_{l=1}^{n-1}\gamma_{l}\,E_{l}^{j}\,E_{l}%
^{i}\right)
\end{align*}
and thus%
\begin{align*}
C^{I}  &  =\sum_{l=1}^{n-1}\left(  E_{l}\wedge\omega\right)  ^{I}%
\frac{\partial}{\partial E_{l}}\left(  \nu\wedge\omega\right)  ^{I}-\sum
_{l=1}^{n-1}\gamma_{l}\left(  \sum_{i\in I}\operatorname*{sgn}\left[
i,I_{\widehat{i}}\right]  E_{l}^{i}\,\omega^{I_{\widehat{i}}}\right)  \left(
\sum_{j\in I}\operatorname*{sgn}\left[  j,I_{\widehat{j}}\right]  E_{l}%
^{j}\,\omega^{I_{\widehat{j}}}\right)  \smallskip\\
&  =\sum_{l=1}^{n-1}\left(  E_{l}\wedge\omega\right)  ^{I}\frac{\partial
}{\partial E_{l}}\left(  \nu\wedge\omega\right)  ^{I}-\sum_{l=1}^{n-1}%
\gamma_{l}\left(  E_{l}\wedge\omega\right)  ^{I}\left(  E_{l}\wedge
\omega\right)  ^{I}%
\end{align*}
which is our claim, since $C=\sum_{I\in\mathcal{T}^{k+1}}C^{I}.$\smallskip

Combining (i) and (ii) we have obtained that%
\[
A=\left\langle \nu\wedge\omega;\nu\wedge\frac{\partial\omega}{\partial\nu
}\right\rangle +\sum_{l=1}^{n-1}\left\langle E_{l}\wedge\omega;\frac{\partial
}{\partial E_{l}}\left(  \nu\wedge\omega\right)  \right\rangle -\sum
_{l=1}^{n-1}\gamma_{l}\left\vert E_{l}\wedge\omega\right\vert ^{2}.
\]

\emph{Step 1.2.} Combining (\ref{tangent equation}) and Step 1.1, we just
proved that%
\begin{align*}
&  \left\Vert d\omega\right\Vert ^{2}+\left\Vert \delta\omega\right\Vert
^{2}-\left\Vert \nabla\omega\right\Vert ^{2}\smallskip\\
&  =\sum_{s=1}^{N}\int_{\Gamma_{s}}\left[  \left\langle \nu\wedge
\omega;d\omega\right\rangle -\left\langle \nu\wedge\omega;\nu\wedge
\frac{\partial\omega}{\partial\nu}\right\rangle -\sum_{l=1}^{n-1}\left\langle
E_{l}\wedge\omega;\frac{\partial}{\partial E_{l}}\left(  \nu\wedge
\omega\right)  \right\rangle +\sum_{l=1}^{n-1}\gamma_{l}\left\vert E_{l}%
\wedge\omega\right\vert ^{2}\right]  .
\end{align*}
Thus, if $\nu\wedge\omega=0$ on $\Gamma_{s}$ for each $s=1,\cdots,N,$ we
obtain (\ref{tangentialcondition}).\smallskip

\emph{Step 2.} Analogous calculations, starting from (\ref{normal k-1 version}%
), establishes the identity%
\begin{align*}
&  \left\Vert d\omega\right\Vert ^{2}+\left\Vert \delta\omega\right\Vert
^{2}-\left\Vert \nabla\omega\right\Vert ^{2}\smallskip\\
&  =\sum_{s=1}^{N}\int_{\Gamma_{s}}\left[  \left\langle \nu\,\lrcorner
\,\omega;\delta\omega\right\rangle -\left\langle \nu\,\lrcorner\,\omega
;\nu\,\lrcorner\,\frac{\partial\omega}{\partial\nu}\right\rangle -\sum
_{l=1}^{n-1}\left\langle E_{l}\,\lrcorner\,\omega;\frac{\partial}{\partial
E_{l}}\left(  \nu\,\lrcorner\,\omega\right)  \right\rangle +\sum_{l=1}%
^{n-1}\gamma_{l}\left\vert E_{l}\,\lrcorner\,\omega\right\vert ^{2}\right]  .
\end{align*}
Using that $\nu\,\lrcorner\,\omega=0$ on $\Gamma_{s}$ for each $s=1,\cdots,N,$
this yields (\ref{normalcondition}). This finishes the proof.\smallskip
\end{proof}

We finally are ready to prove the theorem.\smallskip

\begin{proof}
(Theorem \ref{Theoreme polytopes}). By Theorem \ref{piecewisesmoothformula},
the result is immediate since for a generalized polytope, the principal
curvatures on every face are $0.$\smallskip
\end{proof}

Theorem \ref{piecewisesmoothformula} also immediately implies the following.

\begin{theorem}
Let $\Omega\subset\mathbb{R}^{n}$ be a bounded open Lipschitz set with
piecewise $C^{2}$ boundary, i.e. $\partial\Omega\smallskip=\cup_{s=1}%
^{N}\overline{\Gamma_{s}}\,,$ where the $\Gamma_{s}$ are $C^{2}$ and
relatively open subset of $\partial\Omega$ and $\partial\Omega\setminus
\cup_{s=1}^{N}\Gamma_{s}$\smallskip\ has zero surface measure. If the
principal curvatures are all nonnegative at every point on $\cup_{s=1}%
^{N}\Gamma_{s}\,,$ then%
\[
\left\Vert \nabla\omega\right\Vert ^{2}\leq\left\Vert d\omega\right\Vert
^{2}+\left\Vert \delta\omega\right\Vert ^{2},\quad\forall\,\omega\in C_{T}%
^{1}\left(  \overline{\Omega};\Lambda^{k}\right)  \cup C_{N}^{1}\left(
\overline{\Omega};\Lambda^{k}\right)  .
\]

\end{theorem}

\begin{remark}
Note that unlike the case of smooth domains, here the hypothesis of all
principal curvatures being non-negative does not imply that the domain is
convex. For example, the domain given in polar coordinates by%
\[
\Omega=\left\{  \left(  r,\theta\right)  :\theta_{0}<\theta<2\pi-\theta
_{0},\;r\in\left[  0,1\right)  \right\}  \subset\mathbb{R}^{2},
\]
for some $0<\theta_{0}<\pi/2,$ is a piecewise $C^{2}$ Lipschitz domain which
satisfies the property, but is neither convex, nor can be approximated from
the inside by smooth $1-$convex domains since smooth $1-$convex domains are
necessarily convex. Hence, the result for such domains is not covered by the
result in \cite{Mitrea 2001}.
\end{remark}

For $k=1,$ Theorem \ref{piecewisesmoothformula} also immediately implies, as a
corollary, the following variant of Korn inequality with the precise constant,
which was already observed in Bauer-Pauly \cite{Bauer-Pauly}.

\begin{corollary}
\label{Corollaire Gaffney et Korn}Let $\Omega\subset\mathbb{R}^{n}$ be a
bounded open Lipschitz set with piecewise $C^{2}$ boundary, i.e.
$\partial\Omega\smallskip=\cup_{s=1}^{N}\overline{\Gamma_{s}}\,,$ where the
$\Gamma_{s}$ are $C^{2}$ and relatively open subset of $\partial\Omega$ and
$\partial\Omega\setminus\cup_{s=1}^{N}\Gamma_{s}$\smallskip\ has zero surface
measure. If the principal curvatures are all nonpositive at every point on
$\cup_{s=1}^{N}\Gamma_{s}\,,$ then%
\[
\left\Vert \nabla u\right\Vert ^{2}\leq2\left\Vert \nabla^{sym}u\right\Vert
^{2},\quad\forall\,u\in C_{T}^{1}\left(  \overline{\Omega};\Lambda^{1}\right)
\cup C_{N}^{1}\left(  \overline{\Omega};\Lambda^{1}\right)  ,
\]
where the symmetric gradient $\nabla^{sym}u$ is defined by
\[
\left(  \nabla^{sym}u\right)  _{ij}=\frac{1}{2}\left(  \frac{\partial u^{i}%
}{\partial x_{j}}+\frac{\partial u^{j}}{\partial x_{i}}\right)  ,\quad
\text{for }i,j=1,\cdots,n.
\]

\end{corollary}

\begin{proof}
Theorem \ref{piecewisesmoothformula} in this case implies,%
\[
\left\Vert \nabla u\right\Vert ^{2}\geq\left\Vert \operatorname*{curl}%
u\right\Vert ^{2}+\left\Vert \operatorname*{div}u\right\Vert ^{2},\quad
\forall\,u\in C_{T}^{1}\left(  \overline{\Omega};\Lambda^{1}\right)  \cup
C_{N}^{1}\left(  \overline{\Omega};\Lambda^{1}\right)  .
\]
Integrating the pointwise identity%
\[
\left\vert \nabla u\right\vert ^{2}=\left\vert \nabla^{sym}u\right\vert
^{2}+\frac{1}{2}\left\vert \operatorname*{curl}u\right\vert ^{2}%
\]
and combining with the inequality above gives,%
\[
\left\Vert \nabla^{sym}u\right\Vert ^{2}-\frac{1}{2}\left\Vert \nabla
u\right\Vert ^{2}=\frac{1}{2}\left(  \left\Vert \nabla u\right\Vert
^{2}-\left\Vert \operatorname*{curl}u\right\Vert ^{2}\right)  \geq\frac{1}%
{2}\left\Vert \operatorname*{div}u\right\Vert ^{2}\geq0.
\]
This completes the proof.
\end{proof}

\section{Appendix: an integral identity for differential forms}

We recall that $\Omega\subset\mathbb{R}^{n}$ is a bounded open smooth set with
exterior unit normal $\nu.$ When we say that $E_{1},\cdots,E_{n-1}$ is an
\emph{orthonormal frame field of principal directions} of $\partial\Omega$
with associated principal curvatures $\gamma_{1},\cdots,\gamma_{n-1}\,,$ we
mean that $\left\{  \nu,E_{1},\cdots,E_{n-1}\right\}  $ form an orthonormal
frame field of $\mathbb{R}^{n}$ and, for every $1\leq i\leq n-1,$%
\[
\sum_{j=1}^{n}\left(  E_{i}^{j}\nu_{x_{j}}^{l}\right)  =\gamma_{i}\,E_{i}%
^{l}\quad\Rightarrow\quad\gamma_{i}=\sum_{j,l=1}^{n}\left(  E_{i}^{l}E_{i}%
^{j}\nu_{x_{j}}^{l}\right)  .
\]
The following proposition was used only implicitly and is another version of
the identity obtained in \cite{Csato-Dac 2102} (see also Theorem 5.7 in
\cite{Csato-Dac-Kneuss(livre)}). We first state and prove it in the Euclidean
setting and then for general Riemannian manifolds, using then the notation of
differential geometry and referring to \cite{Csato 2013}.

\begin{theorem}
\label{thm:fund identity with curvatures}Let $0\leq k\leq n.$ Let
$E_{1},\cdots,E_{n-1}$ be an orthonormal frame field of principal directions
of $\partial\Omega$ with associated principal curvatures $\gamma_{1}%
,\cdots,\gamma_{n-1}\,.$ Then every $\alpha,\,\beta\in C^{1}\left(
\overline{\Omega};\Lambda^{k}\right)  $ satisfy the equation%
\begin{align*}
&  \int_{\Omega}\left(  \left\langle d\alpha;d\beta\right\rangle
+\langle\delta\alpha;\delta\beta\rangle-\langle\nabla\alpha;\nabla\beta
\rangle\right)  \smallskip\\
&  =-\int_{\partial\Omega}\left(  \langle\nu\wedge d\left(  \nu\,\lrcorner
\,\alpha\right)  ;\nu\wedge\beta\rangle+\langle\nu\,\lrcorner\,\delta\left(
\nu\wedge\alpha\right)  ;\nu\,\lrcorner\,\beta\rangle\right)  \smallskip\\
&  +\int_{\partial\Omega}\sum_{1\leq i_{1}<\cdots<i_{k}\leq n-1}\left\langle
\alpha;E_{i_{1}\cdots i_{k}}\right\rangle \left\langle \beta;E_{i_{1}\cdots
i_{k}}\right\rangle \sum_{j\in\{i_{1},\cdots,i_{k}\}}\gamma_{j}\smallskip\\
&  +\int_{\partial\Omega}\sum_{1\leq i_{1}<\cdots<i_{k-1}\leq n-1}\left\langle
\alpha;\nu\wedge E_{i_{1}\cdots i_{k-1}}\right\rangle \left\langle \beta
;\nu\wedge E_{i_{1}\cdots i_{k-1}}\right\rangle \sum_{j\notin\left\{
i_{1},\cdots,i_{k-1}\right\}  }\gamma_{j}\,,
\end{align*}
where $E_{i_{1}\cdots i_{k}}=E_{i_{1}}\wedge\cdots\wedge E_{i_{k}}\,.$ In
particular the following two identities hold. For every $\alpha\in W_{T}%
^{1,2}\left(  \Omega;\Lambda^{k}\right)  $%
\[
\int_{\Omega}\left(  \left\vert d\alpha\right\vert ^{2}+\left\vert
\delta\alpha\right\vert ^{2}-|\nabla\alpha|^{2}\right)  =\int_{\partial\Omega
}\sum_{1\leq i_{1}<\cdots<i_{k-1}\leq n-1}\left\vert \left\langle \alpha
;\nu\wedge E_{i_{1}\cdots i_{k-1}}\right\rangle \right\vert ^{2}\sum
_{j\notin\left\{  i_{1},\cdots,i_{k-1}\right\}  }\gamma_{j}%
\]
and for every $\alpha\in W_{N}^{1,2}\left(  \Omega;\Lambda^{k}\right)  $%
\[
\int_{\Omega}\left(  \left\vert d\alpha\right\vert ^{2}+\left\vert
\delta\alpha\right\vert ^{2}-|\nabla\alpha|^{2}\right)  =\int_{\partial\Omega
}\sum_{1\leq i_{1}<\cdots<i_{k}\leq n-1}\left\vert \left\langle \alpha
;E_{i_{1}\cdots i_{k}}\right\rangle \right\vert ^{2}\sum_{j\in\left\{
i_{1},\cdots,i_{k}\right\}  }\gamma_{j}\,.
\]

\end{theorem}

\begin{remark}
If $k=0$ or $k=n,$ then the the right hand side has to be understood as $0$ by
definition. If $k=1,$ the theorem reads as%
\begin{align*}
&  \int_{\Omega}\left(  \left\langle d\alpha;d\beta\right\rangle
+\langle\delta\alpha;\delta\beta\rangle-\langle\nabla\alpha;\nabla\beta
\rangle\right)  \smallskip\\
&  =-\int_{\partial\Omega}\left(  \langle\nu\wedge d\left(  \nu\,\lrcorner
\,\alpha\right)  ;\nu\wedge\beta\rangle+\langle\nu\,\lrcorner\,\delta\left(
\nu\wedge\alpha\right)  ;\nu\,\lrcorner\,\beta\rangle\right)  \smallskip\\
&  +\int_{\partial\Omega}\sum_{i=1}^{n-1}\gamma_{i}\left\langle \alpha
;E_{i}\right\rangle \left\langle \beta;E_{i}\right\rangle +\int_{\partial
\Omega}\left\langle \alpha;\nu\right\rangle \left\langle \beta;\nu
\right\rangle \sum_{j=1}^{n-1}\gamma_{j}%
\end{align*}
and, in particular,%
\[
\int_{\Omega}\left(  \left\vert d\alpha\right\vert ^{2}+\left\vert
\delta\alpha\right\vert ^{2}-|\nabla\alpha|^{2}\right)  =\left\{
\begin{array}
[c]{cl}%
{\displaystyle\int_{\partial\Omega}}
\left(  \gamma_{1}+\cdots+\gamma_{n-1}\right)  \left\vert \left\langle
\alpha;\nu\right\rangle \right\vert ^{2} & \text{if }\alpha\in W_{T}%
^{1,2}\left(  \Omega;\Lambda^{1}\right)  \medskip\\%
{\displaystyle\sum_{i=1}^{n-1}}
{\displaystyle\int_{\partial\Omega}}
\gamma_{i}\left\vert \left\langle \alpha;E_{i}\right\rangle \right\vert ^{2} &
\text{if }\alpha\in W_{N}^{1,2}\left(  \Omega;\Lambda^{1}\right)  .
\end{array}
\right.
\]

\end{remark}

\begin{proof}
The theorem follows from Theorem 5.7 in \cite{Csato-Dac-Kneuss(livre)} and
Lemma \ref{lemma:L nu equal bilinear form with curvatures and dual version}.
The last two identities also follow from (\ref{tangentialcondition}),
respectively (\ref{normalcondition}), of Theorem \ref{piecewisesmoothformula}
and the remark following it.\smallskip
\end{proof}

We now discuss the more general version of the identity. We assume that
$\Omega$ is an $n-$dimensional compact orientable smooth Riemannian manifold
with boundary $\partial\Omega.$ We also adopt the following abbreviations%
\[
\mathcal{T}_{n-1}^{k}=\left\{  I=\left(  i_{1},\cdots,i_{k}\right)
\in\mathbb{N}^{k}:1\leq i_{1}<\cdots<i_{k}\leq n-1\right\}
\]%
\[
E_{I}=\left(  E_{i_{1}},\cdots,E_{i_{k}}\right)  \quad\text{for }%
I\in\mathcal{T}_{n-1}^{k}%
\]%
\[
I^{c}\quad\text{is the complement of $I$ in }\left\{  1,\cdots,n-1\right\}  .
\]

In the next theorem the quantity $F_{k}$ is the linear $0-$th order (not
differential) operator given by the Bochner-Weitzenb\"{o}ck formula, see
\cite{Csato 2013} Section 2.2 for more details. On a flat manifold, for
instance $\mathbb{R}^{n},$ the operator $F_{k}$ is equal to $0$ (which
explains its absence in Theorem \ref{thm:fund identity with curvatures}).

\begin{theorem}
\label{thm:fund identity with curvatures on manifold}Let $0\leq k\leq n.$ Then
every $\alpha,\,\beta\in C^{1}\left(  \overline{\Omega};\Lambda^{k}\right)  $
satisfy the equation%
\begin{align*}
&  \int_{\Omega}\left(  \left\langle d\alpha;d\beta\right\rangle +\left\langle
\delta\alpha;\delta\beta\right\rangle -\left\langle \nabla\alpha;\nabla
\beta\right\rangle \right)  -\int_{\Omega}\left\langle F_{k}\alpha
;\beta\right\rangle \smallskip\\
&  =-\int_{\partial\Omega}\left(  \left\langle \nu\wedge d\left(
\nu\,\lrcorner\,\alpha\right)  ;\nu\wedge\beta\right\rangle +\left\langle
\nu\,\lrcorner\,\delta\left(  \nu\wedge\alpha\right)  ;\nu\,\lrcorner
\,\beta\right\rangle \right)  \smallskip\\
&  +\int_{\partial\Omega}\sum_{I\in\mathcal{T}_{n-1}^{k}}\alpha\left(
E_{I}\right)  \beta\left(  E_{I}\right)  \sum_{j\in I}\gamma_{j}%
+\int_{\partial\Omega}\sum_{I\in\mathcal{T}_{n-1}^{k-1}}\alpha\left(
\nu,E_{I}\right)  \beta\left(  \nu,E_{I}\right)  \sum_{j\in I^{c}}\gamma
_{j}\,.
\end{align*}

\end{theorem}

\begin{remark}
In the case $\alpha=\beta$ a similar form of this identity is known as Reilly
formula, see Theorem 3 in \cite{Raulot-Savo} and the references there. In that
case the identity simplifies: using partial integration twice one obtains, see
Remark 3.8 (iv) in \cite{Csato 2013},%
\[
\int_{\partial\Omega}\left(  \left\langle \nu\wedge d\left(  \nu
\,\lrcorner\,\alpha\right)  ;\nu\wedge\alpha\right\rangle +\left\langle
\nu\,\lrcorner\,\delta\left(  \nu\wedge\alpha\right)  ;\nu\,\lrcorner
\,\alpha\right\rangle \right)  =2\int_{\partial\Omega}\left\langle
\nu\,\lrcorner\,\delta\left(  \nu\wedge\alpha\right)  ;\nu\,\lrcorner
\,\alpha\right\rangle .
\]

\end{remark}

\begin{proof}
The theorem follows from \cite{Csato 2013} Theorem 3.7 and Remark 3.8 (ii) and
from Lemma \ref{lemma:S_k in terms of principal directions} below.\smallskip
\end{proof}

The following Lemma is the analogue of Lemma
\ref{lemma:L nu equal bilinear form with curvatures and dual version}.
$\beta_{N}$ ($\beta_{T}$) denotes the normal (respectively tangential)
component of $\beta$ (see \cite{Csato 2013}).

\begin{lemma}
\label{lemma:S_k in terms of principal directions} Let $S_{k}$ be defined as
in Definition 3.3 in \cite{Csato 2013}. Then the following identities hold%
\begin{align*}
&  \text{(i)}\qquad\left\langle S_{k}\alpha,\beta_{N}\right\rangle =\sum
_{I\in\mathcal{T}_{n-1}^{k-1}}\left[  \alpha\left(  \nu,E_{I}\right)
\beta\left(  \nu,E_{I}\right)  \sum_{j\in I^{c}}\gamma_{j}\right]
\smallskip\\
&  \text{(ii)}\qquad\left\langle S_{n-k}(\ast\alpha),\ast(\beta_{T}%
)\right\rangle =\sum_{I\in\mathcal{T}_{n-1}^{k}}\left[  \alpha\left(
E_{I}\right)  \beta\left(  E_{I}\right)  \sum_{j\in I}\gamma_{j}\right]  .
\end{align*}

\end{lemma}

\begin{proof}
\emph{Step 1.} We first prove (i). Since $E_{1},\cdots,E_{n-1}$ are principal
directions of $\partial\Omega$ they satisfy that%
\[
\nabla_{E_{i}}\nu=\gamma_{i}\,E_{i}\quad\text{and}\quad\left\langle
E_{i};E_{j}\right\rangle =\delta_{ij}\,.
\]
We now choose $\left(  \nu,E_{1},\cdots,E_{n-1}\right)  $ as an orthonormal
basis of $\mathbb{R}^{n}$ to evaluate the scalar product $\left\langle
S_{k}\alpha;\beta_{N}\right\rangle .$ Tuples $\left(  E_{i_{1}},\cdots
,E_{i_{k}}\right)  $ have no contribution in $\left\langle S_{k}\alpha
;\beta_{N}\right\rangle ,$ because $\beta_{N}$ is normal (actually
$S_{k}\alpha$ too). We therefore obtain%
\begin{equation}
\left\langle S_{k}\alpha;\beta_{N}\right\rangle =\sum_{1\leq i_{1}%
<\cdots<i_{k-1}\leq n-1}S_{k}\alpha\left(  \nu,E_{i_{1}},\cdots,E_{i_{k-1}%
}\right)  \beta\left(  \nu,E_{i_{1}},\cdots,E_{i_{k-1}}\right)  ,
\label{eq:Sk alpha beta N}%
\end{equation}
where we have used that $\beta_{N}\left(  \nu,E_{i_{1}},\cdots,E_{i_{k-1}%
}\right)  =\beta\left(  \nu,E_{i_{1}},\cdots,E_{i_{k-1}}\right)  $ by
definition of the normal component. By definition of $S_{k}\,,$ we find%
\begin{align*}
S_{k}\alpha\left(  \nu,E_{i_{1}},\cdots,E_{i_{k-1}}\right)   &  =-\sum
_{j=1}^{n-1}\sum_{l=1}^{k-1}\alpha\left(  E_{j},E_{i_{1}},\cdots,E_{i_{l-1}%
},\text{II}\left(  E_{j},E_{i_{l}}\right)  ,E_{i_{l+1}},\cdots E_{i_{k-1}%
}\right)  \smallskip\\
&  -\sum_{j=1}^{n-1}\alpha\left(  \text{II}\left(  E_{j},E_{j}\right)
,E_{i_{1}},\cdots,E_{i_{k-1}}\right)  \smallskip\\
&  =A_{\left(  i_{1},\cdots,i_{k-1}\right)  }+B_{\left(  i_{1},\cdots
,i_{k-1}\right)  },
\end{align*}
where $\text{II}$ is the second fundamental form. We have also used that
$\left(  E_{r}\right)  _{T}=E_{r}$ for any $r=1,\cdots,n-1,$ since they are
tangent vectors. We therefore get for any $r,s$%
\[
\text{II}\left(  E_{s},E_{r}\right)  =\left\langle \nabla_{E_{s}}E_{r}%
;\nu\right\rangle \nu=-\left\langle E_{r};\nabla_{E_{s}}\nu\right\rangle
\nu=-\gamma_{r}\,\delta_{rs}\,\nu.
\]
This shows that%
\[
B_{\left(  i_{1},\cdots,i_{k-1}\right)  }=\alpha\left(  \nu,E_{i_{1}}%
,\cdots,E_{i_{k-1}}\right)  \sum_{j=1}^{n-1}\gamma_{j}\,.
\]
In the same way we get%
\begin{align*}
A_{(i_{1},\cdots,i_{k-1})}  &  =-\sum_{l=1}^{k-1}\alpha\left(  E_{i_{l}%
},E_{i_{1}},\cdots,E_{i_{l-1}},\text{II}\left(  E_{i_{l}},E_{i_{l}}\right)
,E_{i_{l+1}},\cdots E_{i_{k-1}}\right)  \smallskip\\
&  =\sum_{l=1}^{k-1}\gamma_{i_{l}}\alpha\left(  E_{i_{l}},E_{i_{1}}%
,\cdots,E_{i_{l-1}},\nu,E_{i_{l+1}},\cdots E_{i_{k-1}}\right)  \smallskip\\
&  =-\alpha\left(  \nu,E_{i_{1}},\cdots,E_{i_{k-1}}\right)  \sum_{l=1}%
^{k-1}\gamma_{i_{l}}\,.
\end{align*}
So we obtain that%
\begin{align*}
S_{k}\alpha\left(  \nu,E_{i_{1}},\cdots,E_{i_{k-1}}\right)   &  =\alpha\left(
\nu,E_{i_{1}},\cdots,E_{i_{k-1}}\right)  \left(  \sum_{j=1}^{n-1}\gamma
_{j}-\sum_{l=1}^{k-1}\gamma_{i_{l}}\right)  \smallskip\\
&  =\alpha\left(  \nu,E_{i_{1}},\cdots,E_{i_{k-1}}\right)  \sum_{j\in I^{c}%
}\gamma_{j}\,.
\end{align*}
From this last equation and (\ref{eq:Sk alpha beta N}) the first identity (i)
follows.\smallskip

\emph{Step 2.} We now deduce (ii) from (i) in the following way%
\[
\left\langle S_{n-k}\left(  \ast\alpha\right)  ;\ast\left(  \beta_{T}\right)
\right\rangle =\sum_{I\in\mathcal{T}_{n-1}^{n-k-1}}\left[  \left(  \ast
\alpha\right)  \left(  \nu,E_{I}\right)  \left(  \ast\beta_{T}\right)  \left(
\nu,E_{I}\right)  \sum_{j\in I^{c}}\gamma_{j}\right]  .
\]
Note that the Hodge $\ast$ operator computes on $k-$forms $\omega$ as%
\[
\left(  \ast\omega\right)  \left(  X_{k+1},\cdots,X_{n}\right)  =\omega\left(
X_{1},\cdots,X_{k}\right)
\]
whenever $\left(  X_{1},\cdots,X_{n}\right)  $ is an orthonormal basis of
$\mathbb{R}^{n}.$ We apply this to $X=\left(  \nu,E_{I},E_{I^{c}}\right)  .$
Thus we obtain that for each $I\in\mathcal{T}^{n-k-1}$%
\[
\left(  \ast\alpha\right)  \left(  \nu,E_{I}\right)  \left(  \ast\beta
_{T}\right)  \left(  \nu,E_{I}\right)  =\alpha\left(  E_{I^{c}}\right)
\beta\left(  E_{I^{c}}\right)  .
\]
This leads to, renaming the summation index $I\rightarrow I^{c},$%
\begin{align*}
\left\langle S_{n-k}\left(  \ast\alpha\right)  ;\ast\left(  \beta_{T}\right)
\right\rangle  &  =\sum_{I\in\mathcal{T}_{n-1}^{\left(  n-1\right)  -k}%
}\left[  \alpha\left(  E_{I^{c}}\right)  \beta\left(  E_{I^{c}}\right)
\sum_{j\in I^{c}}\gamma_{j}\right]  \smallskip\\
&  =\sum_{I\in\mathcal{T}_{n-1}^{k}}\left[  \alpha\left(  E_{I}\right)
\beta\left(  E_{I}\right)  \sum_{j\in I}\gamma_{j}\right]  ,
\end{align*}
which proves (ii).\smallskip
\end{proof}

\begin{acknowledgement}
We would like to thank M. Troyanov for providing us with the references
\cite{Alexander 1977} and \cite{Chern-Lashof 1958}. We also want to thank two
anonymous referees for their comments.
\end{acknowledgement}

\end{document}